\documentclass[10pt,oneside,a4paper]{scrartcl}

\usepackage[left=3.5cm,right=3.5cm,top=3cm,bottom=3.5cm,includeheadfoot]{geometry}
\usepackage{german, amsmath, amssymb, amsthm, latexsym, fancyhdr, enumerate, txfonts, expdlist, mdwlist, mathrsfs, tikz, pifont, xspace, cite, hyperref}
\usepackage[active]{srcltx}
\usepackage[T1]{fontenc}
\usepackage[latin1]{inputenc}
\usepackage{subfig}
\renewcommand{\|}{|}

\renewcommand{\ae}{a.e.\@\xspace}

\newcommand{\rrho}{\rho}
\newcommand{\U}{\ensuremath{\mathscr U}}
\newcommand{\F}{\ensuremath{\mathscr F}}

\renewcommand{\P}{\ensuremath{\mathscr P}}
\newcommand{\RR}{R}
\newcommand{\ee}{\delta_1}
\newcommand{\eee}{\delta_2}

\newcommand{\rr}{r_b}
\newcommand{\rrr}{\iota}
\newcommand{\w}{\omega}
\newcommand{\x}{x}

\newcommand{\X}{\ensuremath{X}}

\renewcommand{\d}{\ensuremath{\partial}}

\newcommand{\mc}{\mathcal}

\newcommand{\alphlist}{\begin{list}{(\alph{enumi})}{\usecounter{enumi}}}
\newcommand{\romanlist}{\begin{list}{(\roman{enumi})}{\usecounter{enumi}}}
\newcommand{\listend}{\end{list}}

\renewcommand{\:}{\colon}

\newcommand{\ssq}{\ensuremath{\subseteq}}

\newcommand{\eps}{\ensuremath{\varepsilon}}

\newcommand{\lam}{\ensuremath{\lambda}}

\newcommand{\T}{\ensuremath{\mathbb{T}}}

\newcommand{\N}{\ensuremath{\mathbb{N}}} 
\newcommand{\R}{\ensuremath{\mathbb{R}}}
\newcommand{\Z}{\ensuremath{\mathbb{Z}}}

\newcommand{\I}{\ensuremath{\mathcal I}}

\newcommand{\A}{\ensuremath{\mathcal A}}

\newcommand{\artanh}{\ensuremath{\operatorname{artanh}}}

\newcommand{\Leb}{\ensuremath{\operatorname{Leb}}}

\theoremstyle{plain} 
\newtheorem{thm}{Theorem}[section]
\newtheorem{lem}[thm]{Lemma}
\newtheorem{prop}[thm]{Proposition}

\newtheorem{claim}{Claim}[thm]

\theoremstyle{definition}
\newtheorem{defn}[thm]{Definition}

\theoremstyle{remark}
\newtheorem*{rem}{Remark}

\numberwithin{equation}{section}
\numberwithin{thm}{section}

\makeatletter
\newcommand*{\rom}[1]{\expandafter\@slowromancap\romannumeral #1@}
\renewcommand*{\@fnsymbol}[1]{\ensuremath{\ifcase#1\or \dagger\or *\or \ddagger\or
   \mathsection\or \mathparagraph\or \|\or **\or \dagger\dagger
   \or \ddagger\ddagger \else\@ctrerr\fi}}
\makeatother

\title{Non-smooth saddle-node bifurcations \rom{3}: strange attractors in continuous time}
\author{G. Fuhrmann\thanks{Institute of Mathematics, FSU Jena, Germany. Email: {\tt gabriel.fuhrmann@uni-jena.de}}}
\date{\today}

\begin{document}
\maketitle
\begin{abstract}
Non-smooth saddle-node bifurcations give rise to minimal sets of interesting geometry built of so-called strange non-chaotic attractors.
We show that certain families of quasiperiodically driven logistic differential equations undergo a non-smooth bifurcation.
By a previous result on the occurrence of non-smooth bifurcations in forced discrete time dynamical systems,
this yields that within the class of families of quasiperiodically driven differential equations,
non-smooth saddle-node bifurcations occur in a set with non-empty $\mc C^2$-interior.
\end{abstract}
\section{Introduction}
The logistic differential equation is a model for single-species populations and as such it is certainly among the most famous ode's from mathematical biology.
Much research has been carried out in order to understand the behaviour of a single species in a fluctuating environment both in the applied sciences as well as in pure math 
\cite{Jilson1980,rosenblat1980,bardimartino1983,arrigonisteiner1985,Cushing1986289,vancecoddington1989,doi:10.1080/10236190108808308,
c39ac8e4237042f6a279ada6481f080c,brauersanchez2003,xuboycedaley2005,castilho20071,bravermanmamdani2008,Liu2010730}.
Having in mind that tidal effects--which result from the gravitational interplay of the moon and the sun--are almost surely quasiperiodic,
it is particularly desirable to understand the effect of quasiperiodic forcing when seeking an understanding of long-term ecological behaviour
(see, e.g. \cite{sonechkinivachtchenko2001,spies2007}).
However, up to now, there are few studies taking into account quasiperiodic forcing of the logistic equation in continuous time 
(see the discussion below Theorem~\ref{thm: main existence continuous time}).

An intriguing feature of quasiperiodically forced systems is the occurrence of strange non-chaotic attractors (SNA's).
In the discrete time setting,
the mechanism behind the creation of SNA's as well as their geometry are well-understood \cite{Keller,JagerETDS,JagerAMS,Jager,GrogerJager,fuhrmann2014,fuhrmanngrogerjager14}.
In \cite{fuhrmann2014}, the author has shown that within the class of $\mc C^2$-families of quasiperiodically forced (qpf) monotone interval maps, SNA's occur in a 
set with non-empty interior (see Theorem~\ref{thm: existence of sna discrete time} below).
While the occurrence of strange attractors in systems with low dynamical complexity (in the sense of zero entropy) is with no doubt a fascinating fact in itself,
it seems worth remembering that a main motivation for their study actually came from ode's driven at incommensurate frequencies \cite{GOPY}.
As a matter of fact, the first examples of SNA's happened to be encountered in a 
continuous time setting \cite{Vinograd,millionscikov,koltyzhenkov1987,lipsnitskii} (for a discussion, see also \cite{johnson1982,jorbaetal2007}).
Nonetheless, all examples in the continuous time setting are basically 
projective actions of linear cocycles \cite{Vinograd,millionscikov,koltyzhenkov1987,lipsnitskii,bjerklovcontinuoustime}.\footnote{We should remark that in this linear setting, 
the existence of SNA's is equivalent to the \emph{non-uniform hyperpolicity} of the respective cocycle, a property extensively studied in the 
literature. Results closely related to the ones of the present work can be found (though again, for discrete time systems) in \cite{herman1983,young1997,bjerklov2007}.}

A natural setting for the creation of SNA's are non-smooth saddle-node bifurcations of one-parameter families of 
driven one-dimensional systems (see Section~\ref{sec: saddle-node bif}).
Here, they occur as the outcome of the collision of two continuous invariant curves.
The present work shows that
the property of undergoing such a non-smooth bifurcation has--analogously to the discrete time case--non-empty interior in the $\mc C^2$-topology in the class 
of qpf families of 
one-dimensional ode's (see Theorem~\ref{thm: main existence continuous time}).

The proof of this--in a sense--abstract fact has a consequence which is important from the applied point of view introduced above:
our core idea is to consider the logistic differential equation with quasiperiodic additive forcing and reduce its dynamics--by means of a suitable Poincar\'e section--to those 
of qpf maps that verify the assumptions of Theorem~\ref{thm: existence of sna discrete time}, that is, to maps for which there exists an SNA.
Now, an easy argument shows that the respective reduced system possesses an SNA if and only if the original system does (see Section~\ref{sec: continuous time preliminaries}). 
While the main work thus happens to be the rather technical analysis of the qpf logistic ode (carried out in Section~\ref{sec: logistic ode}), 
the robustness of non-smooth bifurcations in the continuous time case
comes as a by-product of the application of Theorem~\ref{thm: existence of sna discrete time}.
Furthermore, with little extra effort, we can carry over the geometric findings from the discrete time setting in \cite{fuhrmanngrogerjager14} to the present situation 
(see Theorem~\ref{t.main}).

Our main results are contained in Section~\ref{sec: main results}.
Their proofs can be found in the last section of this article.
In the remainder of the current section, we introduce some basic notation and review some facts and definitions from non-autonomous bifurcation theory, fractal geometry
and the discrete time analogue to what we consider in this work.

\paragraph{Acknowledgements.}
I would like to thank Tobias J{\"a}ger for introducing me to the problem and for his helpful remarks on an earlier version of this manuscript.
This work was supported by an Emmy-Noether-Grant of the German Research Council
(DFG grant JA 1721/2-1) and is part of the activities of the Scientific Network ``Skew product dynamics and multifractal analysis'' (DFG grant OE 538/3-1).
\subsection{Setting and Notation}\label{sec: setting and notation}
Throughout this article, let $\X\ssq \R$ be a non-degenerate interval (possibly non-compact), $\T=\R/\Z$, and $D\geq2$.
Given a \emph{non-autonomous vector field}, that is, a map $F\: \T^D\times\X \to \R$,
we study \emph{(local) skew product flows} or, more precisely, \emph{forced one-dimensional (local) flows} of the form
\begin{align}\label{eq: skew product flow preliminaries}
\Xi \: U\ssq \R\times\T^D\times\X \to \T^D\times\X,   \qquad
(t,\theta,x) \mapsto (t\cdot\rho+\theta,\xi(t,\theta,x)),
\end{align}
where $\rho\in \R^D$ and $U$ is the domain of $\xi$ which is the unique (under mild assumptions) maximal solution of
\begin{align}\label{eq: fibre ode general}
\d_t \xi(t,\theta,x)= F\left(t\cdot\rho+\theta, \xi(t,\theta,x)\right)
\end{align}
with $\xi(0,\theta,x)=x$ for each $(\theta,x) \in \T^D\times\X$.
$\T^D$ is called the \emph{base space} or simply \emph{base} of the flow $\Xi$ in \eqref{eq: skew product flow preliminaries}.
We say the differential equation \eqref{eq: fibre ode general} as well as the flow $\Xi$ are \emph{driven by $\rho$}.
Given $\rho$, we may further say $\Xi$ is \emph{generated} by $F$.

We throughout assume that $\rho$ satisfies the following slow recurrence condition.
\begin{defn}\label{def: diophantine continuous}
We say $\rho\in \R^D$ is \emph{Diophantine (of type $(\mathscr C,\eta)$)} if there are $\mathscr C>0$ and $\eta\in \R$ such that
\begin{align*}
 \forall k \in \Z^D \setminus \{0\}\:\left|\sum_{i=1}^D\rho_i k_i \right|\geq \mathscr C |k|^{-\eta}.
\end{align*}
\end{defn}
\begin{rem}
Given $\eta_0>D+1$, it is well-known that almost every $\rho\in \R^D$ is Diophantine of type $(\mathscr C,\eta_0)$ for some $\mathscr C>0$.
\end{rem}

We assume $F$ to be $\mc C^2$ in the following. 
In particular, this implies existence and uniqueness of $\xi(\cdot,\theta,x)$ for all $(\theta,x)\in \T^D\times \X$ (on some non-degenerate time-interval containing $0$)
due to the Picard-Lindel{\"o}f Theorem (see, e.g. \cite[Chapter~\rom{2}, Theorem~1.1 \& 3.1]{Hartman}).
Note that this yields
\begin{align}\label{eq: cocycle property}
 \xi(t+\tau,\theta,x)=\xi(t,\theta+\tau\rho,\xi(\tau,\theta,x)).
\end{align}

Reversing time, we introduce 
\begin{align}\label{eq: xi inverse}
\xi^-\:(t,\theta,x)\mapsto \xi(-t,\theta,x) 
\end{align}
which obviously solves 
\eqref{eq: fibre ode general} with the right-hand side replaced by \begin{align}\label{eq: inverse non-autonomous vectorfield}
F^-\left(t\cdot\rrho^-+\theta, \xi^-(t,\theta,x)\right),
\end{align}
where $\rrho^-=-\rrho$ and $F^-=-F$.
Note that $\xi^-(t,t\cdot\rrho+\theta,\xi(t,\theta,x))=x$ because of \eqref{eq: cocycle property}. 

Although it is standard, we want to mention that throughout this article, we slightly abuse notation by occasionally not distinguishing elements or subsets of $\T^d$ from such of its cover $\R^d$ ($d\in\N$).
For example, we write $|\theta-\theta'|$ for the distance of $\theta,\theta'\in \T^d$ in the metric inherited from the Euclidean norm $|\cdot|$ in $\R^d$.
Further, we identify the tangent space of $\T^d$ at any $\theta$ with $\R^d$.
When speaking of directional derivatives of $\xi(t,\cdot,x)$, we actually have in mind the respective derivatives of a lift of $\xi$ (see \cite[Definition~A.1.19]{KatokHasselblatt}).
In this sense, given $\vartheta\in \R^d\setminus \{0\}$, it is clear what is to be understood by $\d_\vartheta \xi(t,\theta,x)$.
Higher derivatives are denoted and understood analogously.
Typically, we consider directional derivatives with respect to a vector $\vartheta$ with $|\vartheta|=1$ and write $\vartheta \in \mathbb S^{d-1}$ in this case.

\subsection{Non-autonomous saddle-node bifurcations}\label{sec: saddle-node bif}
An \emph{invariant graph} of a local skew-product flow $\Xi$--we might occasionally speak of invariant graphs of $F$ if $F$ is the corresponding non-autonomous vector field and $\rho$ is fixed--is a measurable
function $\phi \:\T^D\to \X$ such that its graph $\Phi=\{(\theta,x)\: x=\phi(\theta)\}$ is invariant under $\Xi$, or equivalently,
\begin{align*}
 \xi\left(t,\theta,\phi(\theta)\right)=
\phi(t \cdot\rrho+\theta) \quad (t\in\R).
\end{align*}
By a slight abuse of notation, we refer by invariant graph to both the map $\phi$ as well as the corresponding point set 
$\Phi=\{(\theta,x)\in\T^D\times \X\:x=\phi(\theta)\}$ which we throughout denote by a capital letter.
We identify invariant graphs which coincide $\Leb_{\T^D}$-almost everywhere,
where $\Leb_{\T^D}$ denotes  Lebesgue measure on $\T^D$.

Associated to each invariant graph $\phi$, there is an invariant ergodic measure $\mu_\phi$ given by
$\mu_\phi(A)=\Leb_{\T^D}(\pi_1(A\cap\Phi))$ for each Borel set $A\ssq\T^D\times\X$, where
$\pi_1:\T^D\times \X\ni(\theta,x)\mapsto \theta$ denotes the canonical projection to the base coordinate.
In fact, the converse is true as well: to each ergodic measure $m$ there is an invariant graph $\phi$ such that $m=\mu_\phi$ (see \cite[Theorem~1.8.4]{arnoldrandomdynamical}). 
Hence, studying the invariant graphs of $\Xi$ amounts to studying its ergodic measures.

Whether an invariant graph $\phi$ attracts or repels nearby orbits, is determined by its \emph{Lyapunov exponent}
\begin{align*}
\lam(\phi)= 1/t\cdot \int_{\T^D}\!\log\left|\d_x \xi\left(t,\theta,\phi(\theta)\right)\right| \,d\theta,
\end{align*} 
which is easily seen to be independent of the particular choice for $t>0$.
If $\lam(\phi)<0$, then $\phi$ is attracting and if $\lam(\phi)>0$, then $\phi$ is repelling (see \cite[Proposition~3.3]{JagerNonLin} for a precise statement).

Denote the set of non-autonomous $\mc C^2$-vector-fields on $\T^D\times \X$ by $\F(\X)$ (keeping the dimension $D$ implicit).
The set of $\mc C^2$-one-parameter families in $\mathscr F(\X)$ is denoted by
\begin{align*}
 \P(\X)=\left\{\left. {\left(F_\beta\right)}_{\beta\in[0,1]} \right | F_\beta\in \, \F(\X) \text{ for all } \beta\in[0,1] \text{ and } (\beta,\theta,x)\mapsto F_\beta(\theta,x) \text{ is } \mc C^2\right\}.
\end{align*}
We may denote elements of $\P(\X)$ also by $\hat {F}={(F_\beta)}_{\beta\in[0,1]}$.
We endow $\P(\X)$ with the 
extended metric
\begin{align*}
 d\big(\hat F,\hat G\big)=\sup_{\substack{(\theta,x)\in \T^D\times \X\\ \beta\in[0,1]}} \sum_{\substack{s_1,s_2,s_3\in\{0,1,2\}\\ s_1+s_2+s_3\leq2}}
\big|\d_\beta^{s_1}\d_\theta^{s_2}\d_x^{s_3}F_\beta(\theta,x)-\d_\beta^{s_1}\d_\theta^{s_2}\d_x^{s_3}G_\beta(\theta,x)\big|.
\end{align*}
With $\tilde d=d/(1+d)$, we may consider $\P(\X)$ a metric space and refer to the respective topology as $\mathcal{C}^2$-topology in all of the following.

In this article, we study local bifurcations of invariant graphs, that is, given $\hat F\in \P(\X)$ we investigate bifurcations of the corresponding graphs 
${(\phi_\beta)}_{\beta\in[0,1]}$ which are assumed to be entirely contained\footnote{We say an invariant graph $\phi$ is \emph{entirely contained} in some set $\Gamma$ if there is a representative $\tilde \phi$
in the equivalence class of $\phi$ whose graph satisfies $\tilde \Phi\ssq\Gamma$.} 
in a compact section $\Gamma=\T^D\times[\gamma^-,\gamma^+]\ssq\T^D\times\X$. 
In particular, we are interested in saddle-node bifurcations, that is, we study ``collisions'' of an attracting with a repelling graph.
A natural setting for these collisions to occur is the subset $\mathscr S(\X)\ssq \P(\X)$ where
each ${(F_\beta)}_{\beta\in[0,1]}\in \mathscr S(\X)$ satisfies the following assumptions for
all $\beta\in[0,1]$ and $\theta\in\T^D$ if applicable
\begin{enumerate}[${(\mathscr S}1)$]
	\item $F_\beta(\theta,\gamma^+) \leq 0 \text{ and } F_\beta(\theta,\gamma^-) \leq  0$;\label{axiom: S0}
 	\item in $\Gamma$, there exist two invariant continuous graphs for $F_0$ but no invariant graph for $F_1$; \label{axiom: S1}
	\item $\d_\beta F_\beta(\theta,x)\leq0$ and there is $\theta_0$ such that $\partial_\beta F_{\beta}(\theta_0,x) < 0$ for all  $x\in[\gamma^-,\gamma^+]$;\label{axiom: S2}
	\item $\d_x^2 F_\beta(\theta,x)<0$ for all $x\in [\gamma^-,\gamma^+]$.\label{axiom: S3}
\end{enumerate}
Here, $(\mathscr S\ref{axiom: S0})$--$(\mathscr S\ref{axiom: S2})$ guarantee that the two initial invariant graphs approach each other monotonously and collide (for a detailed discussion, see \cite{Anagnostopoulou}).
Assumption $(\mathscr S\ref{axiom: S3})$ ensures that there are at most two distinct invariant graphs and that the two invariant graphs of $F_0$ are attracting and repelling, respectively 
\cite[Theorem~2.1]{Anagnostopoulou}.
Note that $\mathscr S(\X)$ has non-empty interior in the $\mc C^2$-topology \cite[Theorem~3.1]{stark1997}.

The next statement describes the possible outcomes of saddle-node bifurcations in the present non-autonomous setting.

\begin{thm}[{cf. \cite[Theorem~3.1]{nunez}, \cite[Theorem~7.1]{Anagnostopoulou}}]
\label{thm: saddle-node}
Fix $\rrho\in \R^D$ and ${(F_\beta)}_{\beta\in[0,1]}\in \mathscr S(\X)$. 
There exists a unique critical parameter $\beta_c\in(0,1)$ such that the following holds.
\romanlist
\item If $\beta<\beta_c$, then $F_\beta$ has two invariant graphs
  $\phi^-_\beta<\phi^+_\beta$ in $\Gamma$, both of which are
  continuous. 
  Further, $\lambda(\phi^-_\beta)>0$ and $\lambda(\phi^+_\beta)<0$. 
\item If $\beta>\beta_c$, then $F_\beta$ has no invariant graphs in
  $\Gamma$.
\item If $\beta=\beta_c$, then one of the following two alternatives
  holds.
\alphlist
\item[\underline{\em Smooth bifurcation:}]  $F_{\beta_c}$ has a
  unique invariant graph $\phi_{\beta_c}$ in $\Gamma$, which satisfies
  $\lambda(\phi_{\beta_c})=0$. Either $\phi$ is continuous, or it contains both
  an upper and lower semi-continuous representative in its equivalence
  class.\footnote{We call an invariant graph \emph{continuous} if it allows for a continuous representative and similarly, we call it
  \emph{semi-continuous} if it allows for a semi-continuous representative. 
  Observe that hence, we call an invariant graph \emph{non-continuous} if there is no 
  continuous representative.}
\item[\underline{\em Non-smooth bifurcation:}] $F_{\beta_c}$ has
  exactly two invariant graphs $\phi^-_{\beta_c}<\phi^+_{\beta_c}$ \ae in
  $\Gamma$. The graph $\phi^-_{\beta_c}$ is lower semi-continuous,
  whereas $\phi^+_{\beta_c}$ is upper semi-continuous, but none of the
  graphs is continuous and there exists a residual set
  $\Omega\ssq\T^d$ such that
  $\phi^-_{\beta_c}(\theta)=\phi^+_{\beta_c}(\theta)$ for all
  $\theta\in\Omega$. \listend \listend
\end{thm}
The graphs appearing in a non-smooth bifurcation are the main theme of this article.
\begin{defn}
 A non-continuous invariant graph $\phi$ is called a \emph{strange non-chaotic attractor (SNA)} if $\lam(\phi)<0$; it is called a \emph{strange non-chaotic repeller (SNR)} if $\lam(\phi)>0$.
\end{defn}
\subsection{Basic notions from fractal geometry}
We will describe the geometry of the SNA's that arise in a non-smooth bifurcation in terms of some concepts from fractal geometry, which we
introduce in this paragraph.

Let $Y$ be a metric space. 
For $\varepsilon>0$,
we call a finite or countable collection $\{A_i\}$ of subsets of $Y$
an {\em $\varepsilon$-cover} of $A$ if $|A_i|\leq\varepsilon$ for each
$i$ and $A\subseteq\bigcup_i A_i$.
\begin{defn}
For $A\subseteq Y$, $s\geq 0$ and $\varepsilon>0$, we define
\[
\mathcal H_\varepsilon^s(A)=\inf\left\{\left.\sum\limits_i
    \left|A_i\right|^s \ \right|\ \{A_i\}\text{ is an
    $\varepsilon$-cover of $A$}\right\}
\]
and call
\[
	\mathcal H^s(A)=\lim\limits_{\varepsilon\to 0} \mathcal H_\varepsilon^s(A)
\]
the \emph{$s$-dimensional Hausdorff measure} of $A$. The
\emph{Hausdorff dimension} of $A$ is defined by
\[
	D_H(A)=\sup\{s\geq 0 \mid \mathcal H^s(A)=\infty\}.
\]
\end{defn}
\begin{rem}
 Notice that $D_H$ is obviously monotone, that is, if $A\subseteq B$, then $D_H(A)\leq D_H(B)$.
\end{rem}

\begin{defn}
  The \emph{lower} and \emph{upper box-counting dimension} of a
  totally bounded subset $A\subseteq Y$ are defined as
\begin{align*}
  \underline D_B(A)=\liminf\limits_{\varepsilon\to 0}\frac{\log N(A,\varepsilon)}{-\log\varepsilon},\\
  \overline D_B(A)=\limsup\limits_{\varepsilon\to 0}\frac{\log
    N(A,\varepsilon)}{-\log\varepsilon},
\end{align*}
where $N(A,\varepsilon)$ is the smallest number of sets of diameter at
most $\varepsilon$ needed to cover $A$.  If $\underline
D_B(A)=\overline D_B(A)$, then we call their common value $D_B(A)$ the
\emph{box-counting dimension} (or \emph{capacity}) of $A$.
\end{defn}
\begin{lem}[{\cite[Corollary~2.4]{falconer2003}}] \label{lem: lipschitz image hausdorff dimension}
  Let $Y$ and $Z$ be two metric spaces and assume that $g:A\subseteq
  Y\to Z$ is a bi-Lipschitz continuous map. Then $D_H(g(A))=D_H(A)$.
\end{lem}
\begin{thm}[{\cite[Corollary~12]{Howroyd1996} \& \cite[Corollary~4]{Howroyd1995}}]\label{thm: Hausdorff dimension product sets}
  Suppose $Y$ and $Z$ are two metric spaces and consider the Cartesian
  product space $Y\times Z$ equipped with the maximum metric. Then for
  $A\subseteq Y$ and $B\subseteq Z$ totally bounded, we have
\begin{align*}
	D_H(A\times B)\leq D_H(A)+\overline D_B(B)\qquad \text{and} \qquad
D_H(A\times B)\geq D_H(A)+D_H(B).
\end{align*}
\end{thm}
\begin{lem}[\cite{falconer2003}]\label{lem: box dimension} 
 Given a totally bounded set with well-defined box-counting dimension $D_B(A)$, we have $D_B(A)=D_B\left(\overline A\right)$.
 Moreover, if $A\ssq \R^d$ and $\textrm{Leb}_{\R^d}(A)>0$, then $D_B(A)=d$.
\end{lem}
\begin{defn}
  For $d\in\N$, we call a Borel set $A\subseteq Y$ \emph{countably
    $d$-rectifiable} if there exists a sequence of Lipschitz
  continuous functions $(g_i)_{i\in\N}$ with $g_i:A_i\subseteq\R^d\to
  Y$ such that $\mathcal H^d(A\backslash\bigcup_i g_i(A_i))=0$. A
  finite Borel measure $\mu$ is called \emph{$d$-rectifiable} if
  $\mu=\Theta\mathcal H^d\!\!\restriction_A$ for some countably
  $d$-rectifiable set $A$ and some Borel measurable density
  $\Theta:A\to[0,\infty)$.
\end{defn}
\subsection{Existence and geometry of SNA's of forced interval maps}\label{sec: skew product maps}
In this section, we review the situation for discrete time systems.
Let us thus consider \emph{qpf monotone interval maps}
 \begin{align}\label{eq: skew-product}
 f\: \T^d\times \X \to \T^d\times \X, \quad\qquad
 (\theta,x)\mapsto \left(\theta+\w,f_{\theta}(x)\right),
\end{align}
where $d\in\N$, $\X$ is as above, $f_{\theta}(\cdot)$ is monotonously increasing for each fixed $\theta$, and $\w\in\T^d$ satisfies the following slow-recurrence condition 
(compare to Definition~\ref{def: diophantine continuous}).
\begin{defn}\label{def: diophantine}
We say $\w\in\T^d$ is \emph{Diophantine (of type $(\mathscr C,\eta)$)} if there are $\mathscr C>0$ and $\eta\in \R$ such that
\begin{align*}
 \forall k \in \Z^d \setminus \{0\}\:\sup_{p\in\Z}\left|\sum_{i=1}^d\w_i k_i+p \right|\geq \mathscr C |k|^{-\eta}.
\end{align*}
\end{defn}

We call $f_{\theta}(\cdot)$ the \emph{fibre map} corresponding to $\theta \in \T^d$. 
The notions of invariant graphs, the associated invariant measures, and SNA's are analogously defined as in the continuous-time case, with
the Lyapunov exponent of an invariant graph $\tilde\phi$ given by
\begin{align*}
\int_{\T^d}\!\log\left|\d_x f_{\theta}\left(\tilde\phi(\theta)\right)\right| \,d\theta.
\end{align*}

\begin{thm}[cf. \cite{fuhrmann2014}]\label{thm: existence of sna discrete time}
Let $\X\ssq\R$ be a non-degenerate interval, suppose $\omega \in \T^d$ is Diophantine and consider the space of one-parameter families
\begin{align*}
{\mc P}_\omega(\X)=\left\{\left(f_\beta\right)_{\beta \in[0,1]}\:  [0,1]\times\T^d\times \X\ni(\beta,\theta,x)\mapsto \left(\theta+\rho,f_{\beta,\theta}(x)\right) \text { is }\mc C^2 \right\}
\end{align*} 
equipped with the $\mc C^2$-metric.\footnote{Which is similarly defined as in the continuous time case.}
There exists a non-empty open set ${\mc U}_\omega(\X)\ssq {\mc P}_\omega(\X)$ such that each ${(f_\beta)}_{\beta\in[0,1]}\in {\mc U}_\omega(\X)$ admits an SNA and an SNR for some $\beta_c\in(0,1)$.
\end{thm}
We specify the set $\mc U_\omega(\X)$ in Section~\ref{sec: logistic ode} (see Theorem~\ref{thm: sink-source orbit refined}).
It is essentially given by a number of $\mc C^2$-estimates on the map $f$ restricted to a section $\T^d\times[\gamma^-,\gamma^+]$ (with $f_\beta(\gamma^\pm)\leq \gamma^\pm$) 
that contains the SNA/SNR-pair of the previous statement.

We next provide a finer geometric description of the SNA $\tilde\phi_{\beta_c}^+$ that occurs in the previous theorem and its corresponding ergodic measure $\mu_{\tilde\phi_{\beta_c}^+}$.
We moreover give a simple description of the minimal set in the section $\Gamma=\T^d\times [\gamma^-,\gamma^+]$ 
as the \emph{maximal invariant set (in $\Gamma$)}. 
For $\beta\in[0,1]$, this is given by
\begin{equation*}
  \tilde\Lambda_\beta \ = \ \bigcap_{n\in \Z} f^n_{\beta}(\Gamma).
\end{equation*}
Note that $\tilde\Lambda_{\beta_c}\neq\emptyset$. 
It turns out that
\begin{align*}
  \psi_{\tilde\Lambda_{\beta_c}}^-(\theta) \ = \inf \tilde\Lambda_{\beta_c}(\theta) \quad\textrm{ and } 
  \quad \psi_{\tilde\Lambda_{\beta_c}}^+(\theta) \ = \ \sup\tilde\Lambda_{\beta_c}(\theta) 
\end{align*}
are lower and upper semi-continuous representatives of $\tilde\phi_{\beta_c}^+$ and $\tilde\phi_{\beta_c}^-$.\footnote{The semi-continuity
is a general fact (see \cite{stark2003}), the other statement follows from the definition of $\mc U_\w(\X)$ (see \cite{phdfuhrmann2015})}

The next assertion is a crucial ingredient for our geometric analysis.
\begin{prop}[cf. \cite{fuhrmanngrogerjager14}]\label{prop: lipschitz}
Let $\omega$ and ${(f_\beta)}_{\beta\in[0,1]}$ be as in Theorem~\ref{thm: existence of sna discrete time}.
There is an increasing sequence of sets $\tilde\Omega_j\ssq\T^d$ such that $\psi_{\tilde\Lambda_{\beta_c}}^+\!\!\!\restriction_{\tilde\Omega_j}$ is Lipschitz continuous for all $j\in\N$ and
$D_H(\tilde\Omega_\infty)\leq d-1$, where $\tilde\Omega_\infty=\T^d\setminus\bigcup_{j\in\N}\tilde\Omega_j$.
An analogous result holds for the repeller $\psi_{\tilde\Lambda_{\beta_c}}^-$.
\end{prop}

\begin{thm}[cf. \cite{fuhrmanngrogerjager14}]\label{t.main discrete}
  Let $\omega$ and ${(f_\beta)}_{\beta\in[0,1]}$ be as in Theorem~\ref{thm: existence of sna discrete time}. Then the SNA $\psi_{\tilde\Lambda_{\beta_c}}^+$ satisfies the following.
\begin{enumerate}[(i)]
\item 
  $D_B\left(\Psi_{\tilde\Lambda_{\beta_c}}^+\right)=d+1$ and
  $D_H\left(\Psi_{\tilde\Lambda_{\beta_c}}^+\right)=d$.
\item The measure $\mu_{\psi_{\tilde\Lambda_{\beta_c}}^+}$ 
is $d$-rectifiable.
\item $\tilde\Lambda_{\beta_c}$
  is minimal. We have
  $\tilde\Lambda_{\beta_c}=\overline{\Psi_{\tilde\Lambda_{\beta_c}}^-}=
  \overline{\Psi_{\tilde\Lambda_{\beta_c}}^+}$.
\item $\psi_{\tilde\Lambda_{\beta_c}}^+$ 
is the only semi-continuous representative in its equivalence class. 
\end{enumerate}
  Analogous results hold for the repeller $\psi_{\tilde\Lambda_{\beta_c}}^-$.
\end{thm}
\begin{rem}
 Part (ii) as well as the statement concerning the Hausdorff dimension are direct implications of Proposition~\ref{prop: lipschitz} (cf. \cite[Theorem~3.2]{fuhrmanngrogerjager14}).
\end{rem}

\section{Main results}\label{sec: main results}
The primary work of this article is to show the following.
\begin{thm}\label{thm: main 1}
Let $h$ be a non-increasing $\mc C^2$-bump-function $h\:\R_{\geq 0}\to [0,1]$ with $h'(0)=0$, $h'\!\!\restriction_{(0,\RR)}<0$ 
(for some $\RR>0$), $h''(0)<0$, $h(y)=0$ for all $y\geq\RR$, and $h(0)=1$.
Given Diophantine $\rho\in\R^D$ of type $(\mathscr C,\eta)$, there is $\mathscr R=\mathscr R(|\rho|,\mathscr C,\eta)$ such that the family of qpf skew product flows ${(\Xi_{\beta})}_{\beta\in[0,1]}$ 
driven by $\rho$ and generated by the non-autonomous vector fields
\[
F_\beta(\theta,x)=-bx^2+b-\beta b/(1-b^{-1/2}) \cdot h(\|\theta\|),
\]
undergoes a non-smooth saddle-node bifurcation within $\T^D\times [-1,1]$ if $R\geq \mathscr R$ and $b$ is sufficiently large.
\end{thm}
In fact, we prove more: we show that for large enough $b$, the family ${(\Xi_{\beta})}_{\beta\in[0,1]}$ can be reduced to 
a family of qpf monotone interval maps which lies in the set $\mc U_\w(\X)$ of Theorem~\ref{thm: existence of sna discrete time} (for appropriate $\w$ and $\X$ [see Section~\ref{sec: poincare map}]).
This yields the following corollary.
\begin{thm}\label{thm: main existence continuous time}
Suppose $\rho\in\R^D$ is Diophantine. Then there is a set
$ \U_\rho(\X) \, \ssq\, \P(\X)$ with non-empty interior in the
$\mathcal{C}^2$-topology such that each family of flows ${(\Xi_\beta)}_{\beta\in[0,1]}$ driven by $\rho$ and generated 
by some ${(F_\beta)}_{\beta\in[0,1]}\in \U_\rho(\X)$ undergoes a non-smooth saddle-node bifurcation.
\end{thm}
Coming back to the particular vector fields of Theorem~\ref{thm: main 1}, observe that we immediately get 
an analogous result for the skew product flow family generated by
\[
 L_\beta(\theta,x)=2/r\cdot bx\cdot(r-x)-\beta b/(1-b^{-1/2})  \cdot h(\|\theta\|),
\]
for some $r>0$. 
In other words, Theorem~\ref{thm: main 1} indeed guarantees the occurrence of a non-smooth saddle-node bifurcation for
the logistic differential equation with a (certain) quasiperiodic harvesting term.

In \cite{bjerklovcontinuoustime}, a similar Riccati equation as in Theorem~\ref{thm: main 1} is considered.
There as well, the existence of SNA's is proven, however, in a regime already beyond the saddle-node bifurcation.
Hence, in the words of the situation considered in \cite{bjerklovcontinuoustime}, 
the novelty of Theorem~\ref{thm: main 1} is that we describe the \emph{transition} from uniform to non-uniform hyperbolicity.
On a technical level, this difference is most visibly reflected in the fact that in the present work we have to cope
with second derivatives of the flow.

Note that by the application of Theorem~\ref{thm: existence of sna discrete time},
our approach focusses on the geometry of the mechanism by which SNA's are created.
This geometric insight shows that
even if in general, analytical results on the occurrence of non-smooth bifurcations for particular ode's might still be subject to rather technical considerations, 
the proof in Section~\ref{sec: logistic ode} should basically be extendable to situations with 
non-autonomous vector fields similar to the one above (for example, we believe that it should be possible to replace $h$ by an arbitrary $\mc C^2$-function [on $\T^D$]
with a unique non-degenerate global maximum).
In a nutshell, it is not so much the particular choice of the vector fields $F_\beta$ but the assumption of general 
features--like the concavity of the functions $F_\beta(\theta,\cdot)$ and the decreasing dependence on 
$\beta$--which guarantee a non-smooth saddle-node bifurcation (see Figure~\ref{fig: SNA}).
\begin{figure} 
\centering
\begin{tabular}{c}
\vspace{-35pt} 
\hspace{-20pt}
\subfloat{\includegraphics[scale=0.7]{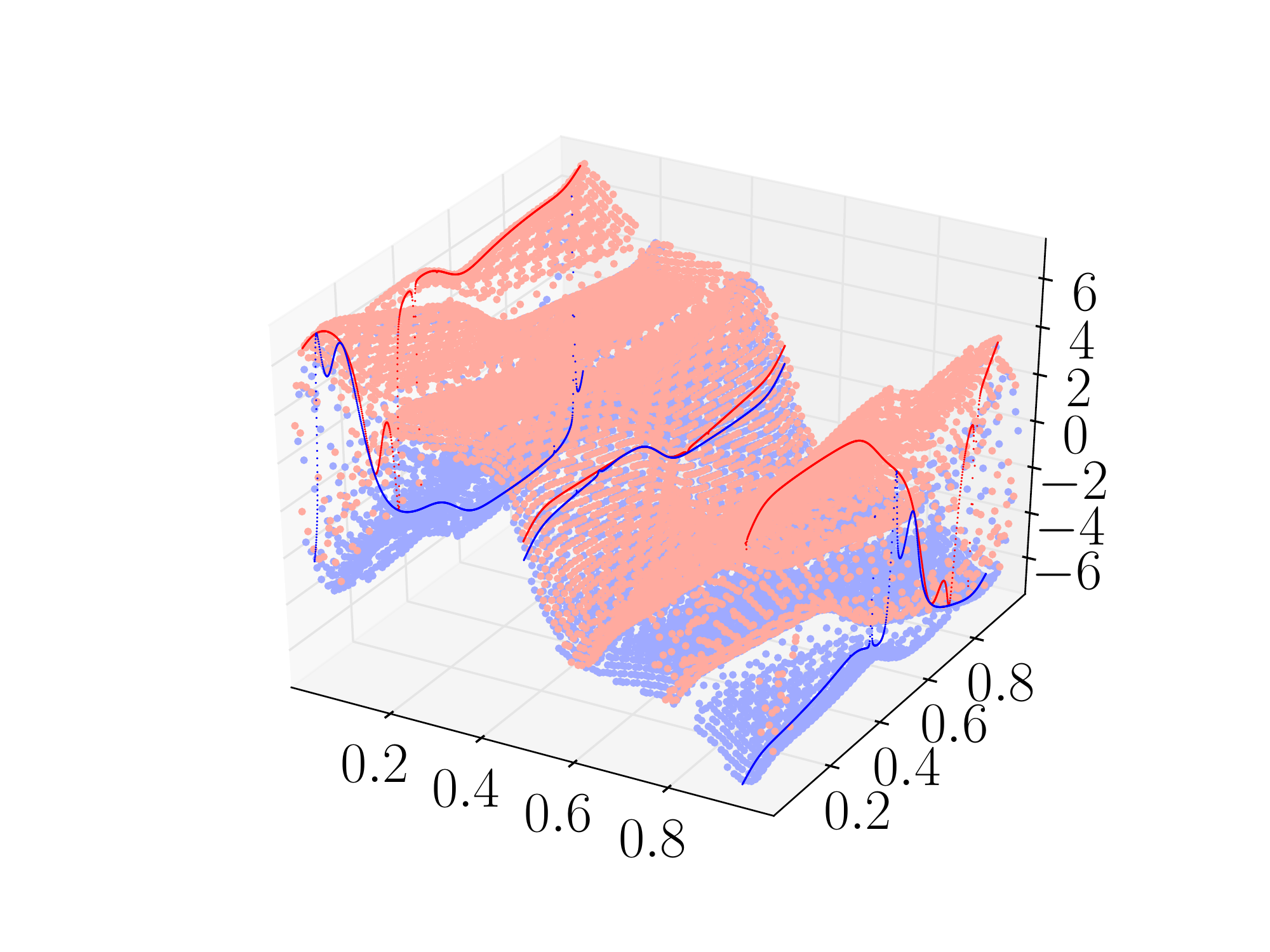}}\\
\raisebox{10mm}{\subfloat{\includegraphics[scale=0.4]{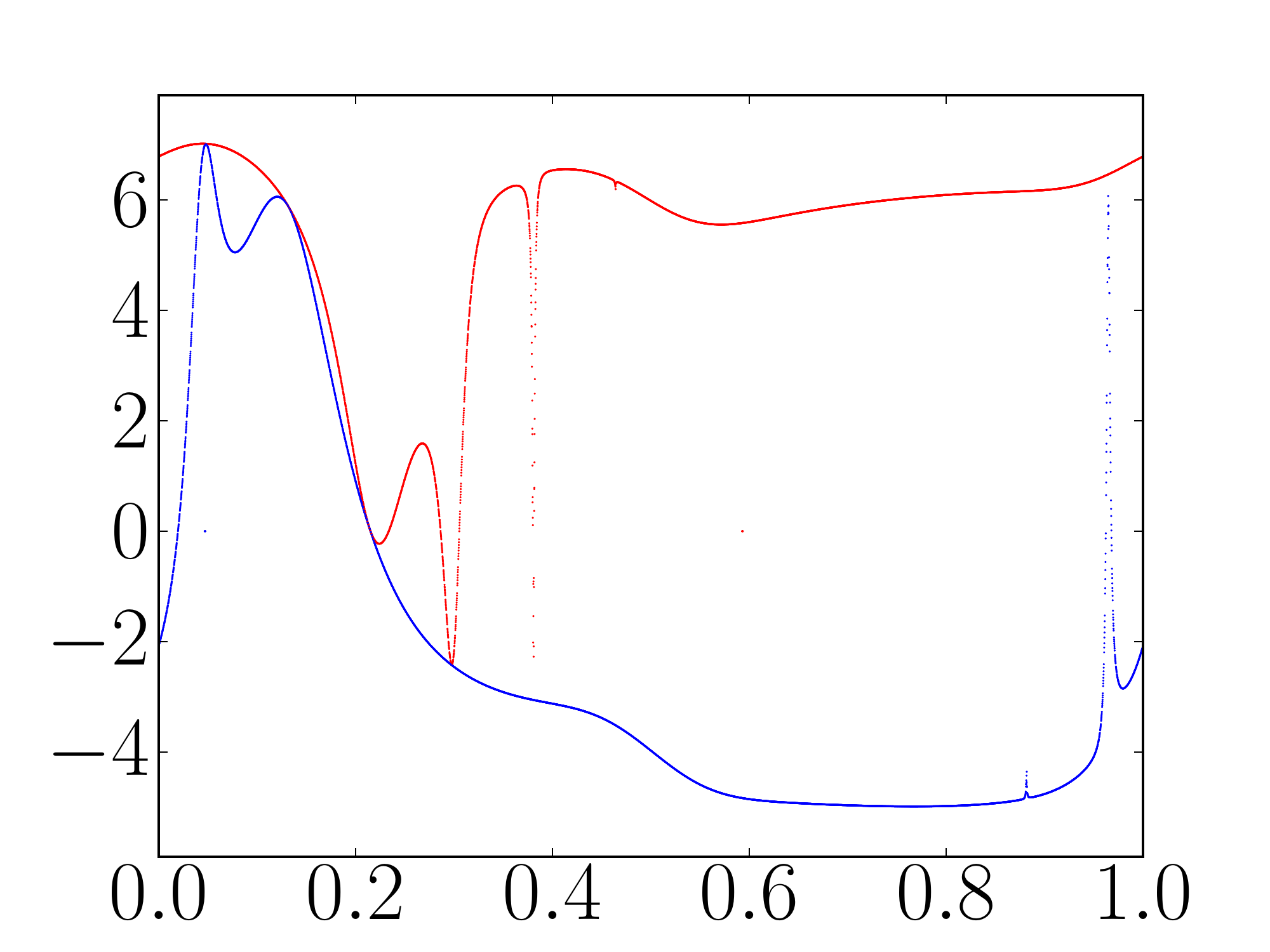}}}
\vspace{-10pt}
\end{tabular}
\caption{Invariant graphs of a skew product flow with $D=2$ close to a non-smooth saddle-node bifurcation. 
The considered family of vector fields is given by $F_\beta(\theta,x)=-x^2+b-\beta\cdot(2-\cos^{11}(2\pi \theta_1)-\cos^{11}(2\pi \theta_2))/4$ where $\theta =(\theta_1,\theta_2)$.
We put $\rrho=((\sqrt{5} - 1)/2,\pi)$, $b=100$, and $\beta=176.01538$.
In the upper picture, the attractor is pale red, the repeller is depicted in pale blue. 
The curves in deep red and blue show slices of the attractor and repeller, respectively, for three different fixed values of $\theta_1$. 
The lower picture allows a closer look at the section close to $\theta_1=0$.} 
\label{fig: SNA}
\end{figure}

Finally, another merit of the reduction to the discrete time setting is that we can--with only a little extra work--carry over Theorem~\ref{t.main discrete} to the continuous time setting.
Hence, we obtain a fairly comprehensive description of the geometry of the SNA and the maximal invariant set in $\Gamma=\T^D\times [\gamma^-,\gamma^+]$ (recall that we are dealing with local bifurcations).
For $\beta\in[0,1]$, this is--similarly to the discrete time case--given by
\begin{equation*}
  \Lambda_\beta \ = \ \bigcap_{t\in\R} \Xi_{\beta}(t,\Gamma).
\end{equation*}
Note that $\Lambda_\beta$ is non-empty for $\beta\leq \beta_c$ due to Theorem~\ref{thm: saddle-node}.
Analogously to the discrete time case, we have that
\begin{align*}
  \phi_{\Lambda_\beta}^-(\theta) \ = \inf \Lambda_{\beta}(\theta) \quad\textrm{ and } 
  \quad \phi_{\Lambda_\beta}^+(\theta) \ = \ \sup\Lambda_{\beta}(\theta) 
\end{align*}
are lower and upper semi-continuous invariant graphs, respectively, and hence representatives of the invariant 
graphs that appear along the saddle-node bifurcation of ${(\Xi_\beta)}_{\beta\in[0,1]}$.

\newpage
\begin{thm}\label{t.main}
  Let $\rho$ and ${(\Xi_\beta)}_{\beta\in[0,1]}$ be as in Theorem~\ref{thm: main existence continuous time}. 
  Then the SNA $\phi_{\Lambda_{\beta_c}}^+$ appearing at the critical parameter $\beta_c$ satisfies the following.
\begin{enumerate}[(i)]
\item 
  $D_B\left(\Phi_{\Lambda_{\beta_c}}^+\right)=D+1$ and
  $D_H\left(\Phi_{\Lambda_{\beta_c}}^+\right)=D$.
\item The measure $\mu_{\phi_{\Lambda_{\beta_c}}^+}$ 
is $D$-rectifiable.
\item $\Lambda_{\beta_c}$
  is minimal. We have  $\Lambda_{\beta_c}=\overline{\Phi_{\Lambda_{\beta_c}}^-}=
  \overline{\Phi_{\Lambda_{\beta_c}}^+}$. \label{item: 3 geometric continuous}
\item $\phi_{\Lambda_{\beta_c}}^+$ 
is the only semi-continuous representative in its equivalence class. 
\end{enumerate}
  Analogous results hold for the repeller $\phi_{\Lambda_{\beta_c}}^-$.
\end{thm}
\begin{rem}
Note that $D$-rectifiability of a measure $\mu$ implies that $\mu$
is exact dimensional with the point-wise dimension equal to $D$ \cite{AmbrosioKirchheim2000}.
As a result of this, several other dimensions of $\mu$ coincide
with $D$ \cite{Zindulka2002}. In
particular, this is true for the information dimension
\cite{Young1982Dimension}.
\end{rem}
\section{Prerequisites}\label{sec: continuous time preliminaries}
As already pointed out, the overall strategy of this article is to reduce particular skew product flows to 
qpf monotone interval maps in order to extend both Theorem~\ref{thm: existence of sna discrete time} and Theorem~\ref{t.main discrete} to qpf ode's.
Appropriate Poincar\'e sections (and their corresponding return maps) which are suitable for this reduction are introduced in the second paragraph of this section.

As the application of the discrete time results involves to deal with a number of $\mc C^2$-estimates, we 
carry out some straightforward computations to provide the derivatives of $(\beta,t,\theta,x)\mapsto\xi_\beta(t,\theta,x)$
in the next paragraph.

\subsection{Derivatives of the flow}
It is well known (see, e.g. \cite[Chapter~\rom{5}, Corollary~4.1]{Hartman}) that for $\hat F\in \P(\X)$,
the map $(\beta,t,\theta,x)\mapsto\xi_\beta(t,\theta,x)$ is $\mc C^2$ so that the task here is to differentiate \eqref{eq: fibre ode general} and express the solutions of
the resulting ode's (sometimes referred to as \emph{variational equations}) in terms of the (unknown) solution $\xi_\beta$ of \eqref{eq: fibre ode general}.
By differentiating (\ref{eq: fibre ode general}) with respect to $x$ and $\vartheta$, we get
\begin{align}
 \d_t \d_x \xi_\beta(t,\theta,x)&=
\d_x F_\beta\left(t\rrho+\theta, \xi_\beta(t,\theta,x)\right) \cdot  \d_x \xi_\beta(t,\theta,x), \label{eq: d x fibre ode}\\
\d_t \d_\vartheta \xi_\beta(t,\theta,x)&= \d_\vartheta F_\beta\left(t\rrho+\theta, \xi_\beta(t,\theta,x)\right)+
\d_x F_\beta\left(t\rrho+\theta, \xi_\beta(t,\theta,x)\right) \cdot  \d_\vartheta \xi_\beta(t,\theta,x). \label{eq: d theta fibre ode}
\end{align}
Further, note that since $\xi_\beta(0,\theta,x)=x$, we have 
$\d_x \xi_\beta(0,\theta,x)=1$ and
$\d_\vartheta \xi_\beta(0,\theta,x)=0$.
Hence,
\begin{align}
 \d_x \xi_\beta(t,\theta,x) =& \, \exp\left(\int_0^t\!
\d_x F_\beta\left(s\rrho+\theta, \xi_\beta(s,\theta,x)\right)\, ds\right) \label{eq: d x xi}
\end{align}
and 
\begin{align}
\label{eq: d theta xi}
\d_\vartheta \xi_\beta(t,\theta,x)
= 
\int_0^t \! \d_\vartheta F_\beta\left(s\rrho+\theta, \xi_\beta(s,\theta,x)\right)
\exp\left(\int_s^t\!
\d_x F_\beta\left(\tau\rrho+\theta, \xi_\beta(\tau,\theta,x)\right)\, d\tau\right)
\, ds .
\end{align}
The expression for
$\d_x \xi_\beta(t,\theta,x)$ immediately shows monotonicity of $\xi_\beta$ in $x$. However, observe that this already 
follows from the uniqueness of the solutions of \eqref{eq: fibre ode general}, of course.
We can differentiate \eqref{eq: d x xi} to get
\begin{align}
\label{eq: d2 x xi}
\d^2_x \xi_\beta(t,\theta,x) 
= &\, \d_x\xi_\beta(t,\theta,x)\cdot
\int_0^t\!
\d^2_x F_\beta\left(s\rrho+\theta, \xi_\beta(s,\theta,x)\right) \cdot \d_x \xi_\beta(s,\theta,x)\, ds
\end{align}
and similarly
\begin{align}
\label{eq: dx dvartheta xi}
\begin{split}
&\d_\vartheta \d_x \xi_\beta(t,\theta,x)\\
&= \d_x \xi_\beta(t,\theta,x)\cdot
\!\!\int_0^t\! 
\d_\vartheta\d_x F_\beta\left(s\rrho+\theta,\xi_\beta(s,\theta,x)\right)+
\d_x^2F_\beta\left(s\rrho+\theta,\xi_\beta(s,\theta,x)\right)\cdot \d_\vartheta\xi_\beta(s,\theta,x)\,ds.
\end{split}
\end{align}
For simplicity, instead of further differentiating \eqref{eq: d theta xi}  with respect to $\vartheta$,
we differentiate \eqref{eq: d theta fibre ode} and solve the resulting problem with initial condition $\d^2_\vartheta \xi_\beta(0,\theta,x)=0$
in order to obtain an expression for $\d^2_\vartheta \xi_\beta(t,\theta,x)$. 
Now,
\begin{align*}
 \d_t \d_\vartheta^2 \xi_\beta(t,\theta,x)&= \d_\vartheta^2 F_\beta\left(t\rrho+\theta, \xi_\beta(t,\theta,x)\right)+
2\d_\vartheta\d_x F_\beta\left(t\rrho+\theta, \xi_\beta(t,\theta,x)\right) \cdot  \d_\vartheta \xi_\beta(t,\theta,x)\\
&\phantom{=}+\d_x^2 F_\beta\left(t\rrho+\theta, \xi_\beta(t,\theta,x)\right) \cdot  \left(\d_\vartheta \xi_\beta(t,\theta,x)\right)^2\\
&\phantom{=}+\d_x F_\beta\left(t\rrho+\theta, \xi_\beta(t,\theta,x)\right) \cdot  \d_\vartheta^2 \xi_\beta(t,\theta,x).
\end{align*}
The solution is straightforwardly given by
\begin{align}
\nonumber
& \d^2_\vartheta \xi_\beta(t,\theta,x)=\\
\label{eq: d2 theta xi}
&\int_0^t
\left [\d^2_x F_\beta\left(\!s\rrho\!+\!\theta, \xi_\beta(s,\theta,x)\right) 
\left(\d_\vartheta \xi_\beta(s,\theta,x)\right)^2
+\d^2_\vartheta F_\beta\left(\!s\rrho\!+\!\theta, \xi_\beta(s,\theta,x)\right) \right.
\\\nonumber
& \left. \vphantom{\left(\d_\vartheta \xi_\beta(s,\theta,x)\right)^2} +
2
\d_x\d_\vartheta F_\beta\left(\!s\rrho\!+\!\theta, \xi_\beta(s,\theta,x)\right)
\d_\vartheta \xi_\beta(s,\theta,x)\right ]
\exp\left(\int_s^t\!
\d_x F_\beta\left(\tau\rrho+\theta, \xi_\beta(\tau,\theta,x)\right)\, d\tau\right)
\, ds.
\end{align}
\subsection{Reduction to a Poincar\'e map}\label{sec: poincare map}
Let us drop the index $\beta$ in this paragraph and set $d=D-1$.
 
Assume without loss of generality that $|\rrho_D|=\max_{j=1,\ldots,D}|\rrho_j|$.
Note that since $\rrho=\left(\rrho_1,\ldots,\rrho_D\right)\in \R^D$ is Diophantine of type $(\mathscr C,\eta)$ 
(see Definition~\ref{def: diophantine continuous}), we have that $\w=\w(\rho)=(\rrho_1/\rrho_D,\ldots,\rrho_{d}/\rrho_D)\in \T^d$ is 
Diophantine of type $({\mathscr C}',\eta')$ 
(see Definition~\ref{def: diophantine})
with $\eta'=\eta$ and $\mathscr C'$ proportional to $\mathscr C/\rrho_D$.

Before we can reduce the \emph{local} flow $\Xi$ from \eqref{eq: skew product flow preliminaries} to a skew product of
the form \eqref{eq: skew-product}, we need to make it a flow, that is, we need the set $U$ to equal $\R\times\T^D\times\X$ so that
any point in $\T^D\times\X$ has a full orbit.
To that end, recall that we are dealing with local bifurcations occurring in a section $\Gamma=\T^D\times[\gamma^-,\gamma^+]\ssq \T^D\times \X$ and assume, for simplicity,  that $[\gamma^-,\gamma^+]$ is in the interior of $\X$.
By changing the non-autonomous vector field $F$ outside of $\Gamma$, we obviously do not change anything about the considered bifurcation within $\Gamma$.
Hence, we may replace $F$ by $\tilde F=h \cdot F$, where $h\:\Theta\times\X\ni (\theta,x)\mapsto \tilde h(x) \in[0,1]$ is a smooth function with $\tilde h\!\!\restriction_{[\gamma^--\eps,\gamma^++\eps]}=1$ and 
$\tilde h\!\!\restriction_{\X\setminus[\gamma^--2\eps,\gamma^++2\eps]}=0$
for some $\eps>0$ with $[\gamma^--3\eps,\gamma^++3\eps]\ssq\X$.
With $\tilde F$, we actually have a flow since a given orbit either stays within
$[\gamma^--2\eps,\gamma^++2\eps]$ or is eventually constant so that every orbit is well-defined for all times.
In the following, we will not stress this detail.
However, the reader should always think of the modified vector field $\tilde F$ whenever full orbits are assumed for arbitrary initial conditions.

In this sense, consider the first return map to the Poincar\'e section $\T_D^d=\{\theta\in \T^{D}\:\theta_D=0\}$, that is, the map
\begin{align}\label{eq: first return map}
\begin{split}
\tilde \Xi \: \T_D^d\times\X &\to \T_D^d\times\X,\\
(\theta,x) &\mapsto \Xi(1/\rrho_D,\theta,x)=\left(\theta+1/\rho_D\cdot \rho,\tilde\xi_{\theta}(x)\right),
\end{split}
\end{align}
where $\tilde\xi_{\theta}(x)=\xi(1/\rrho_D,\theta,x)$. Note that
\eqref{eq: first return map} is of the form \eqref{eq: skew-product}. 
From now on, we identify $\T_D^d$ with $\T^d$ and thus consider $\T^d$ a subset of $\T^D$ (slightly abusing terminology). 
In a similar fashion, we may write $\theta+\w$ when $\theta \in\T^D$ and actually $\theta+1/\rrho_D \cdot \rrho$ is meant.

It is obvious that an invariant graph of the flow $\Xi$ yields an invariant graph for its first return map $\tilde \Xi$.
The following basic observation provides us with a converse.
\begin{prop}\label{prop: invariant graph for time-1 map is invariant for the flow}
Consider the flow $\Xi$ in (\ref{eq: skew product flow preliminaries}) with a non-autonomous $\mc C^1$-vector field $F$ and suppose there is an invariant graph $\tilde\phi\: \T^d\to X$ for the corresponding first return map $\tilde\Xi$.
Then there is a unique invariant graph $\phi$ for $\Xi$ with $\phi(\theta)=\tilde\phi(\theta)$ for each $\theta\in \T^d$ and $\phi$ is continuous if and only if $\tilde\phi$ is continuous. Further,
if ${\tilde\Phi}$ is relatively compact in $\T^d\times\X$, then
$\lam(\phi)=\rrho_D\cdot\lam(\tilde\phi)$.
\end{prop}
\begin{proof}
For $\theta\in \T^D$, set $t_\theta=\theta_D/\rrho_D$. Then, $\phi\:\theta\mapsto
\xi\left(t_\theta,\theta-t_\theta\rrho,\tilde\phi(\theta-t_\theta\rrho)\right)$ is invariant under $\Xi$. The uniqueness and the assertion about the continuity are obvious. 

Now, note that if $\overline{\tilde\Phi}$ is compact in $\X$, then so is $\overline{\Phi}\ssq\Xi\left([0,1]\times\overline{\tilde\Phi}\right)$. 
As $F$ is $\mc C^1$, $\d_x F$ is thus bounded on $\overline{\Phi}$ and hence integrable. 
By means of \eqref{eq: d x xi}, we therefore have
\begin{align*}
\lam(\tilde\phi)&=\int_{\T^d}\!\log\left|\d_x \tilde\xi_\theta\left(\tilde\phi(\theta)\right)\right|\,d\theta= 
\int_{\T^d}\!\int_0^{1/\rrho_D}\!\d_x F_\beta\left(\theta+s\rrho,\xi_\beta(s,\theta,\tilde\phi(\theta)\right)\, ds \,d\theta\\
&=
\int_{\T^d}\!\int_0^{1/\rrho_D}\!\d_x F_\beta\left(\theta+s\rrho,\phi(\theta+s\rrho)\right)\, ds \,d\theta=
1/\rrho_D\cdot\int_{\T^D}\!\d_x F_\beta\left(\theta,\phi(\theta)\right)\, d\theta
\end{align*}
and hence
\begin{align*}
\lam(\phi)&=\rrho_D \cdot\int_{\T^D}\!\log\left|\d_x \xi\left(1/\rho_D,\theta,\phi(\theta)\right)\right|\,d\theta\\
&=
\rrho_D\cdot\int_{\T^1}\!\int_{\T^d}\int_0^{1/\rrho_D}\!\d_x F_\beta\left(\theta+s\rrho,\phi(\theta+s\rrho)\right)\, ds \,d(\theta_1,\ldots,\theta_d)\,d\theta_D
\\
&=
\int_{\T^1} \int_{\T^D}\!\d_x F_\beta\left(\theta,\phi(\theta)\right)\, d\theta \, d\theta_D=\rrho_D\cdot\lam(\tilde\phi). \qedhere
\end{align*}
\end{proof}
\section[The quasiperiodically driven logistic differential equation]{The quasiperiodically driven logistic differential equation\sectionmark{The driven logistic differential equation}}\label{sec: logistic ode}
\sectionmark{The driven logistic differential equation}
Let us consider a one-parameter family of skew product flows $\Xi_\beta$  
of the form (\ref{eq: skew product flow preliminaries}) with $\X=\R$ and
\begin{align}\label{eq: radial x-square family}
\tag{$\ast$} 
F_\beta(\theta,x)=-bx^2+b-\beta b/\left(1-b^{-1/2}\right)\cdot g(\theta),
\end{align}
where $b>1$ and $g\:\T^D\to [0,1]$ is $\mc C^2$ and assumes a unique non-degenerate global maximum at some $\overline \theta\in \T^D$. 
Without loss of generality, we may assume that $g\left(\overline\theta\right)=1$.
To reduce the technicalities of our investigation to a minimum, we assume further that $g(\theta)=h(|\theta-\overline\theta|)$, 
where $h$ is a non-increasing $\mc C^2$-bump-function $h\:\R_{\geq 0}\to [0,1]$ with $h'(0)=0$, $h'\!\!\restriction_{(0,\RR)}<0$ 
(for some $\RR>0$), $h''(0)<0$, $h(y)=0$ for all $y\geq\RR$, and $h(0)=1$.

It is not hard to see that ${(F_\beta)}_{\beta\in [0,1]}$ lies in $\P(\R)$ and satisfies $(\mathscr S\ref{axiom: S0})$--$(\mathscr S\ref{axiom: S3})$ with
$\gamma^+=1$ and $\gamma^-=-1$ 
(where $(\mathscr S\ref{axiom: S1})$ follows from Claim~\ref{claim: crossing is possible in continuous time} below), that is, \eqref{eq: radial x-square family} undergoes a saddle-node bifurcation in the sense of Theorem~\ref{thm: saddle-node}.
In fact, this is true for any section containing $\T^D\times[\gamma^-,\gamma^+]$.
Our goal is to show that there is $\mathscr R$ (independent of $b$) such that if $R\geq \mathscr R$ and $b$ is sufficiently large, then \eqref{eq: radial x-square family} undergoes a \emph{non-smooth} bifurcation.

By the previous section, we hence have to show that the first return maps ${(\tilde \xi_\beta)}_{\beta\in[0,1]}$ corresponding to \eqref{eq: radial x-square family} lie in the set $\mc U_\w(\X)$ of
Theorem~\ref{thm: existence of sna discrete time}.
Recall that we have not explicitly defined the set $\mc U_\w(\X)$ so far.
It is essentially given by a list of estimates, denoted by $(\mc A1)$--$(\mc A15)$, on the fibre maps and their first as well as second derivatives.
In the following, we will provide these estimates already adapted to the first return maps ${(\tilde \Xi_\beta)}_{\beta\in[0,1]}$ and concurrently prove that they are actually verified in the present situation.
Note that there are subtle differences to the presentation in \cite{fuhrmann2014}.
The reader interested in the details may consult \cite{phdfuhrmann2015}.

Let us introduce some notation.
Given $\theta$ and $\theta'$ in $\T^D$ such that there is $s\in[0,1/\rrho_D]$ with
$\theta'=\theta+s\rrho$, we denote by $[\theta,\theta']$ the line segment $\{\theta+\tau\rrho\:\tau\in[0,s]\}$.
Similarly, given $A,B\ssq \T^D$ such that for all $\theta\in A$ there exists  a unique $s(\theta)\in[0,1/\rrho_D]$ so 
that $B=\{\theta+s(\theta)\rrho\:\theta\in A\}$, we set $[A,B]=\bigcup_{\theta\in A}[\theta,\theta+s(\theta)\rrho]$.

We suppose there is $\ee>0$ much smaller than $1/\rrho_D$ (in fact, $\ee<\min\{ 1/18,1/(36\rrho_D)\}$ is sufficient) such that
$[\T^d,\T^d-\ee\rrho]\cap B_\RR\left(\overline \theta\right)=\emptyset$ where $B_\RR\left(\overline \theta\right)$ denotes the ball of radius $\RR$ around $\overline \theta$, that is, in one iteration, the time span before an orbit hits the bump is 
much bigger than the time after hitting the bump.
We may further assume there is a positive constant $\eee<\ee$ such that $[\T^d-\eee\rrho,\T^d]\cap B_\RR\left(\overline \theta\right)=\emptyset$.
By possibly shifting the $\theta_D=0$ section, both assumptions boil down to assuming that $R$ is small (independently of $b$).

To establish the existence of an attracting and a repelling invariant graph,
we need regions where the dynamics are contracting and expanding, respectively.
We will locate these regions in an interval of contraction $C=[1-c,1+c]$ (for some positive $c<1/4$) and an interval of expansion $E=[-1,-1+\exp(-b/(2\rrho_D))]$.\footnote{Observe that by choosing
$C$ to lie above $E$, we decided the repeller to be below the attractor.}
In the following, we restrict our analysis to the section $\Gamma=\T^d\times[-1,1+c]$.

Although in principle, we could consider the flows corresponding to \eqref{eq: radial x-square family} for all $\beta\in[0,1]$, we will show below that if $\beta$ is too close to $1$,
there are $\theta$ and $x$ such that there exist $t_-<t_+ \in \R$ with $\lim_{t\to t_{\pm}}\Xi_\beta(t,\theta,x)=\mp \infty$.
Such solutions clearly rule out the existence of an invariant graph in $\Gamma$. 
For that reason, setting
\begin{align*}
 \beta_-&=\min \{\beta \in [0,1]\: \exists \theta\in \T^d \text{ such that } \tilde\xi_{\beta,\theta}(1-c)\leq -1+\exp(-b/[2\rrho_D]) \},\\
\beta_+&=\max\{\beta \in [0,1]\:   \tilde\xi_{\beta,\theta}(1+c)\geq -1 \text{ for all } \theta\in \T^d \},
\end{align*}
all of the following assumptions are only supposed to hold for all $\beta \in[0,\beta_+]$ (if applicable).

Finally, we define the \emph{critical region}
\begin{align*}
\I_0=\left\{\theta\in \T^d\: [\theta,\theta+\w]\cap \overline{B_\RR\left(\overline \theta\right)}\neq \emptyset\right\}
\end{align*}
and introduce the following constants which will serve as bounds on our derivatives
\begin{align*}
\alpha_c^{-1}=\alpha_e=\exp\left[2b(1-c) (1/\rrho_D-\ee)-10b\ee\right]
\, \quad \text{and} \quad \,
\alpha_l^{-1}=\alpha_u=\exp[2b(1+c)/\rrho_D].
\end{align*}

With these definitions, we are now in a position to go through the assumptions that define the set $\mc U_\w(\X)$.
Let us first consider ${(\mathcal A}1)$--${(\mathcal A}8)$.
\begin{enumerate}[${(\mathcal A}1)$]
\item $0<\d_x \tilde\xi_{\beta,\theta}(x)<\alpha_c$ for $(\theta,x)\in \T^d\times C$;\label{axiom: 1c}
\item $\d_x \tilde\xi_{\beta,\theta}(x)>\alpha_e$ for $(\theta,x)\in (\T^d\setminus \mc I_{0})\times E\bigcap \tilde\Xi_\beta^{-1}(\T^d\times E)$;\label{axiom: 2c}
\item $\alpha_l<\d_x \tilde\xi_{\beta,\theta}(x)<\alpha_u$ for all $(\theta,x)\in \Gamma\cap \tilde\Xi_\beta^{-1}(\Gamma)$.
\label{axiom: 3c}
\suspend{enumerate}
Observe that the above assumptions justify the naming of the intervals $C$ and $E$.

The mechanism by which the SNA/SNR-pair is created in Theorem~\ref{thm: existence of sna discrete time} is essentially the following: 
First, the existence of a continuous attractor and a continuous repeller is guaranteed for small $\beta$.
Then, by increasing $\beta$, we move these two initial invariant graphs closer and closer to each other on a small set until they finally touch on a $\Leb_{\T^d}$-null set. This yields the desired discontinuity.
The existence of the initial invariant graphs is ensured by
\resume{enumerate}[{[${(\mathcal A}1)$]}]
\item $\tilde\xi_{\beta,\theta}(1+c)\leq 1+c$ and $\tilde\xi_{\beta,\theta}(-1)\leq -1$ \label{axiom: 4c}
\suspend{enumerate}
and the first part of ${(\mathcal A}\ref{axiom: 5c})$ below. 
In order to make the invariant graphs touch each other, we have to connect the regions of contraction and expansion more and more with growing $\beta$ which is why we assume a certain monotonicity
in the dependence of $\beta$ (see ${(\mathcal A\ref{axiom: 6c})}$).
The connection between $C$ and $E$, however, needs to be realised carefully in order to guarantee that the two graphs only touch on a measure zero set.
To that end, we only allow orbits starting within the critical region to pass from $C$ to $E$
\resume{enumerate}[{[${(\mathcal A}1)$]}]
\item $\tilde\xi_{\beta,\theta}\left(x\right)\in C$ for all $x\in [-1+\exp(-b/(2\rrho_D)),1+c]$ and $\theta \notin \I_{0}$. \label{axiom: 8c}
\suspend{enumerate}
The set which contains (at least) all $\theta$ for which ${\tilde\xi}_{\beta,\theta}(1-c)\leq -1+\exp(-b/(2\rrho_D))$ is denoted by $\mc J_{0,\beta}$ and is obviously a subset of $\I_{0}$.
We want it to satisfy
\resume{enumerate}[{[${(\mathcal A}1)$]}]
\item $\mc J_{0,\beta}$ is closed and convex and $\mc J_{0,\beta}\ssq \mc J_{0,\beta'}$ for each $\beta'\geq\beta$; \label{axiom: 10c}
\item $\d_\vartheta^2 \tilde\xi_{\beta,\theta}(x)>s$ for each $\vartheta \in {\mathbb S}^{d-1}$, $x\in C$ and all $\theta \in \mc J_{0,\beta}$.
\label{axiom: 9c}.
\item $\tilde\xi_{0,\theta}(1-c)\geq 1-c$ for all $\theta\in \T^d$ and 
$\tilde\xi_{\beta_+,\theta}(1+c) \leq -1$ for some $\theta \in \T^d$; \label{axiom: 5c}
\item $\tilde\xi_{(\cdot)}(\theta,\x)$ is non-increasing for fixed $(\theta,x)\in \Gamma$. \label{axiom: 6c}
\suspend{enumerate}

Before we come to prove $(\mathcal A\ref{axiom: 1c})$--$(\mathcal A\ref{axiom: 6c})$, we provide some simple observations.
From now on, given $\rrho$, we denote by $\theta_0 \in\T^d$ that point which passes through the maximum of $g$ in $\overline\theta$ within one time step, that is, $\overline \theta \in [\theta_0,\theta_0+\w]$.
\begin{prop}\label{prop: radial basic observations}
Suppose $F_\beta$ is given by \eqref{eq: radial x-square family}.
Then
\begin{enumerate}[(a)]
\item $\xi_{\beta}(t,\theta,x)<x$ if $|x|>1$ where $t>0$, $\theta \in \T^D$, 
and $\beta\in [0,1]$.
\label{label a prop radial basic observations}
 \item $\xi_{\beta}(t,\theta,x)\geq x$ if $|x|\leq1$ where $\theta \in \T^D$, $t\in[0,1/\rho_D]$ is such that $[\theta,\theta+t\rrho]\cap B_\RR\left(\overline\theta\right)=\emptyset$, and $\beta\in [0,1]$.
\label{label b prop radial basic observations}
\end{enumerate}
\end{prop}
\begin{proof}
	\eqref{label a prop radial basic observations} follows easily from the fact that $F_\beta(\theta,x)< 0$ for arbitrary $\beta$, $\theta$, and $|x|>1$.
	 Likewise, \eqref{label b prop radial basic observations} follows from the fact that $F_\beta(\theta,x)\geq0$ for arbitrary $\beta$, $\theta\notin B_\RR\left(\overline\theta\right)$, and $|x|\leq1$.
\end{proof}
Now, let us consider $(\mc A \ref{axiom: 1c})$--$(\mc A \ref{axiom: 6c})$ in opposite order. 
$(\mc A \ref{axiom: 6c})$ follows immediately from the monotone dependence of \eqref{eq: radial x-square family} on $\beta$. 
The first part of $(\mc A \ref{axiom: 5c})$ is immediate.
The existence of $\beta_+\in(0,1)$ such that the second estimate of $(\mc A \ref{axiom: 5c})$ holds true follows from the next statement under the assumption of sufficiently large $b$.

It is convenient to introduce $U_\eps=\{\theta \in \T^D\: g(\theta)\geq1-\eps\}$ where $\eps>0$.
Clearly, $U_\eps$ is nothing but $B_{h^{-1}([1-\eps,1])}\left(\overline \theta\right)$.
\begin{claim}\label{claim: crossing is possible in continuous time}
Suppose $F_\beta$ is given by \eqref{eq: radial x-square family} and $b$ is sufficiently large.
Then there exists $\beta \in (0,1)$ such that 
$\xi_{\beta}(t,\theta_0,1+c)$ is well-defined for all $t\in[0,1/\rrho_D]$ and
$\tilde \xi_{\beta,\theta_0}(1+c)\leq-1$.
\end{claim}
\begin{proof}
	Note that for $t$ with $\theta_0+t\rrho \in U_{b^{-1/2}/2}$, we have $\d_t \xi_{\beta}(t,\theta_0,1+c)= -b\xi_{\beta}(t,\theta_0,1+c)^2+ b-\beta b/(1-b^{-1/2}) g(\theta+t\rrho)\leq
	-b\xi_{\beta}(t,\theta_0,1+c)^2-b(\beta/(2b^{1/2}-2)+\beta-1)$. 
        Consider such $\beta$ for which
	 $\beta/(2b^{1/2}-2)+\beta-1>0$.
	Observe that
	\begin{align*} y_\beta(t)=-\sqrt{\beta/(2b^{1/2}-2)+\beta-1}\tan\left(b\sqrt{\beta/(2b^{1/2}-2)+\beta-1}(t-t_0)+\alpha\right)
	\end{align*}
	 with $\alpha=\arctan\left(-(1+c)/\sqrt{\beta/(2b^{1/2}-2)+\beta-1}\right)$ solves 
\[\d_t y_\beta(t)=-b y_\beta(t)^2-b(\beta/(2b^{1/2}-2)+\beta-1)\]
	  with $y_\beta(t_0)=1+c$.
	  Thus, $y_\beta(t)$ is
	an upper bound for $\xi_{\beta}(t,\theta_0,1+c)$ for all $t\in[t_0,t_1]$, where $[t_0,t_1]$ is set to be the maximal interval with  
	$[\theta_0+t_0\rrho,\theta_0+t_1\rrho]\ssq U_{b^{-1/2}/2}$.
	Note that $|t_1-t_0|>b^{-1/2}$ for big enough $b$ since $g$ assumes a maximum in $\overline \theta$.
	Now, for large enough $b$, there is $\beta\in(0,1)$ such that $y_\beta(b^{-1/2}+t_0)<-1$ which proves that the image of $[0,1)\ni\beta\mapsto \xi_\beta(t_1,\theta_0,1+c)$ contains $[-1,1]$. Proposition~\ref{prop: radial basic observations}\eqref{label a prop radial basic observations} and the continuous dependence of $\tilde \xi_{\beta,\theta_0}(1+c)$ on $\beta$ hence yield the statement.
\end{proof}
$(\mc A\ref{axiom: 10c})$ and $(\mc A\ref{axiom: 9c})$ are treated in Lemma~\ref{lem: nondegenerate bump}. For sufficiently large $b$, $(\mc A\ref{axiom: 8c})$ is a consequence of the following statement.
\begin{claim}
Suppose $F_\beta$ is given by \eqref{eq: radial x-square family}, $\theta\notin \I_0$ and $b$ is sufficiently large. Then $\tilde\xi_{\beta,\theta}(-1+\exp[-b/(2\rrho_D)])>1-c$.
\end{claim}
\begin{proof}
Note that as $\theta\notin \I_0$, we have that $\xi_\beta(t,\theta,-1+\exp(-b/(2\rrho_D)))$
equals $y(t)$ for $t\in[0,1/\rrho_D]$, where $y$ is the solution of the initial value problem
\begin{align*}
 \dot y =-b y^2+b, \quad y(0)=-1+\exp(-b/(2\rrho_D)).
\end{align*}
 Now, $y(t)=\tanh(b\cdot t+\alpha)$, where $|\alpha| =|\!\artanh(-1+\exp[-b/(2\rrho_D)])|\leq b/(3\rrho_D)$. Hence, $y(1/\rrho_D)\geq\tanh(2b/(3\rrho_D))>1-c$ for large enough $b$. 
\end{proof}
$(\mc A \ref{axiom: 4c})$ is an immediate consequence of Proposition~\ref{prop: radial basic observations} \eqref{label a prop radial basic observations}. 
Hence, apart from $(\mc A\ref{axiom: 10c})$ and $(\mc A\ref{axiom: 9c})$, we are left with $(\mc A \ref{axiom: 1c})$--$(\mc A \ref{axiom: 3c})$ which follow from the next assertion.
\begin{prop}\label{prop: expansion in E contraction in C}
Suppose $F_\beta$ is given by \eqref{eq: radial x-square family} and $b$ is sufficiently large. Then
\begin{enumerate}[(a)]
 \item $\d_x \tilde\xi_{\beta,\theta}(x)\leq\exp(-2b(1-c) (1/\rrho_D-\ee)+4b\ee)$ and $\tilde\xi_{\beta,\theta}(x)>-2$ for $(\theta,x)\in \T^d\times C$; \label{label: claim a}
\item $\d_x \tilde\xi_{\beta,\theta}(x)\geq\exp(2b(1-\exp[-b/(2\rrho_D)])/\rrho_D)$ for $(\theta,x)\in (\T^d\setminus \I_0)\times E\cap {\tilde\Xi}^{-1}_\beta (\T^d\times E)$; 
\label{label: claim b}
\item  $\exp(-2b(1+c)/\rrho_D) <\d_x \tilde\xi_{\beta,\theta}(x)\leq\exp(2b/\rrho_D)$ for 	all $(\theta,x)\in \Gamma\cap {\tilde\Xi}^{-1}_\beta (\Gamma)$. \label{label: claim c}
\end{enumerate}
\end{prop}
\begin{proof}
Note that due to Proposition~\ref{prop: radial basic observations}\eqref{label a prop radial basic observations}, we have that $(\theta,x)\in\Gamma\cap {\tilde\Xi}^{-1}_\beta (\Gamma)$ necessarily implies $\xi_\beta(t,\theta,x)\in [-1,1+c)$ for all $t\in(0,1/\rho_D]$.

Now, \eqref{label: claim c} follows from equation \eqref{eq: d x xi} since we have $2b\geq\d_xF_\beta>-2b(1+c)$ on $\T^D\times[-1,1+c)$.
Similarly, we obviously get item \eqref{label: claim b} as long as $\xi_\beta(t,\theta,x)\in E$ for all $t\in[0,1/\rrho_D]$ which necessarily holds 
for $(\theta,x)\in (\T^d\setminus \I_0)\times E\bigcap {\tilde\Xi}^{-1}_\beta (\T^d\times E)$ due to Proposition~\ref{prop: radial basic observations} \eqref{label b prop radial basic observations}.

It remains to consider \eqref{label: claim a}.
Note that for all $x\in C$, all $\beta \in [0,1]$, each $\theta\in \T^d$, and $t\in[0,1/\rrho_D-\ee]$ we have $\xi_\beta(t,\theta,x)\geq \xi_\beta(t,\theta,1-c)\geq1-c$.
Suppose for a contradiction there were $\theta'\in \T^d$ and  
$\beta\in[0,\beta_+]$ such that $\tilde \xi_{\beta,\theta'}(1-c)=-2$. 
Note that in this case $\xi_{\beta}(t,\theta',1-c)\geq-2$ holds necessarily for all $t\in[0,1/\rrho_D]$ because
of Proposition~\ref{prop: radial basic observations}\eqref{label a prop radial basic observations}.
Thus, \eqref{eq: d x xi} yields 
\begin{align}\label{eq: upper bound contraction rate}
\begin{split}
\d_x \tilde\xi_{\beta,\theta'}(x)&=\exp\left(\int_0^{1/\rrho_D}\!
\d_x F_\beta(s\rrho+\theta', \xi_\beta(s,\theta',x))\, ds\right)\\
&=
\exp\left(-2b\int_{0}^{1/\rrho_D-\ee}\!
 \xi_\beta(s,\theta',x)\, ds-2b\int_{1/\rrho_D-\ee}^{1/\rrho_D}\!
 \xi_\beta(s,\theta',x)\, ds\right)\\
&\leq \exp(-2b (1-c) (1/\rrho_D-\ee)+4b\ee)
\end{split}
\end{align}
for all $x\in[1-c,1+c]$. 
Hence in this case, we had $\tilde \xi_{\beta_+,\theta'}(1+c)\leq \tilde \xi_{\beta,\theta'}(1+c)\leq-2+2c\cdot \exp(-2b (1-c)(1/\rrho_D-\ee)+4b\ee)<-1$ (where the last inequality holds for large enough $b$ if we assume $\ee<1/(36\rrho_D)$)
contradicting the definition of $\beta_+$. 
Thus, we have $\tilde \xi_{\beta,\theta}(x)>-2$ for all $\theta\in \T^d$, $x\in C$ and $\beta\in[0,\beta_+]$ and hence \eqref{eq: upper bound contraction rate} yields
an upper bound for $\d_x \tilde\xi_{\beta,\theta}(x)$ with $(\theta,x)\in\T^d\times C$.
\end{proof}
The core part of this article is to show that there is $\mc J_{0,\beta}$ and $s>0$ such that $(\mc A \ref{axiom: 10c})$ and $(\mc A \ref{axiom: 9c})$ hold.
The respective proof is given with the next lemma which sheds light on the mechanism by which the sensitive interplay of contraction and expansion is realised.
In short words, the problem is to seize control over the second derivatives of solutions of an ode by means of its right-hand side.
\begin{lem}\label{lem: nondegenerate bump}
Suppose $F_\beta$ is given by \eqref{eq: radial x-square family}, $R$ is small enough, $b$ is sufficiently large, 
and $\beta\in[\beta_-,\beta_+]$.

Then there is $\mc J_{0,\beta}\ssq\mc I_{0}$ such that $\d^2_\vartheta \tilde\xi_{\beta,\theta}(x)>\exp(b\eee/4)$ for arbitrary $x\in C$, $\vartheta \in \mathbb S^{d-1}$, and $\theta\in \mc J_{0,\beta}$.
Further, $\mc J_{0,\beta}$ contains all $\theta\in\T^d$ with $\tilde\xi_{\beta,\theta}(1-c)\leq-1+\exp(-b/(2\rrho_D))$ and $(\mc A \ref{axiom: 10c})$ is satisfied.
\end{lem}
For later use, we provide some crude and straightforward estimates in the following auxiliary statement.
We denote by $\mathbf{1}_{A}$ the characteristic function of a set $A\ssq\T^D$, that is, $\mathbf{1}_{A}=1$ on $A$ and $\mathbf{1}_{A}=0$ on
$\T^D\setminus A$.
\begin{claim}\label{claim: int exp (int F(xi)) is small}
For $(\theta,x) \in \T^d\times C$, $\beta\in[0,\beta_+]$, $t_1\in[0,1/\rrho_D-\ee]$, and $t\in[t_1,1/\rrho_D]$, we have under the assumption of sufficiently large $b$ that
\begin{align}\label{eq: upper bound exp int}
 &\int_0^{t-t_1}\!
\exp\left(\int_s^{t-t_1}\!
\d_x F_\beta\left((\tau+t_1)\rrho+\theta,\xi_\beta(\tau+t_1,\theta,x)\right)\, d\tau\right)
\, ds \leq \exp(5b\ee).
\end{align}
Further, suppose ${\tilde\xi}_{\beta,\theta}(x)\leq-3/4$ and $t\geq1/\rrho_D-\eee/2$. There is $\RR_0<\RR$ such that for sufficiently large $b$
\begin{align}\label{eq: lower bound exp int}
&\int_{0}^{t-t_1}\!
\mathbf{1}_{B_{\RR_0}\left(\overline\theta\right)}((s+t_1)\rrho+\theta)
\exp\left(\int_s^{t-t_1}\!
\d_x F_\beta\left((\tau+t_1)\rrho+\theta,\xi_\beta(\tau+t_1,\theta,x)\right)\, d\tau\right)
\, ds
\geq \exp(b\eee/2).
\end{align}
Finally, if $0\leq t_1\leq 1/\rrho_D-5\ee$ and $t\in[t_1,1/\rrho_D]$, then
we have for all $(\theta,x) \in \T^d\times C$ and sufficiently large $b$ that
\begin{align}
 \label{eq: dx xi <= 1}
\exp\left(\int_0^{t-t_1}\!
\d_x F_\beta\left((\tau+t_1)\rrho+\theta,\xi_\beta(\tau+t_1,\theta,x)\right)\, d\tau\right) \leq 1.
\end{align}

\end{claim}
\begin{proof}[Proof of the claim.]
The relations can be seen in a similar fashion as \eqref{eq: upper bound contraction rate}.
In particular, we make use of the fact that $\xi_\beta(\tau+t_1,\theta,x)\geq-2$ for all $\tau\in[0,1/\rrho_D-t_1]\supseteq[0,t-t_1]$ 
and $\xi_\beta(\tau+t_1,\theta,x)\geq 1-c$ for $\tau\leq 1/\rrho_D-\ee-t_1$.
For $x\in C$, this implies
\begin{align*}
 \exp\left(\int_s^{t-t_1}\!\d_x F_\beta\left((\tau+t_1)\rrho+\theta,\xi_\beta(\tau+t_1,\theta,x)\right)\, d\tau\right)
=\exp\left(-2b\int_s^{t-t_1}\! \xi_\beta(\tau+t_1,\theta,x)\,d\tau\right)
\leq
 \exp(4b\ee) 
\end{align*}
such that
\begin{align*}
 \int_0^{t-t_1}\!
\exp\left(\int_s^{t-t_1}\!
\d_x F_\beta\left((\tau+t_1)\rrho+\theta,\xi_\beta(\tau+t_1,\theta,x)\right)\, d\tau\right)
\, ds \leq 1/\rrho_D\cdot \exp(4b\ee), 
\end{align*}
which is smaller than $\exp(5b\ee)$ for big enough $b$.

For the second inequality, note that there is $0<\tilde \RR<\RR$ such that for big enough $b$ we have $F_\beta(\theta,-3/4)\geq0$ for all $\theta\notin B_{\tilde \RR}\left(\overline\theta\right)$ and all $\beta$. 
Hence, for all $\theta$ and $x$ as in the assumptions,
we necessarily have that $\xi_{\beta}(\tau,\theta,x)\leq-3/4$ 
for all 
$\tau\in[0,1/\rrho_D]$ with $[\theta+\tau\rrho,\theta+\w]\cap B_{\tilde\RR}\left(\overline\theta\right)=\emptyset$. 
Set $\RR_0=(\RR+\tilde\RR)/2$.
Then,
\begin{align*}
& \int_{0}^{t-t_1}\!
\mathbf{1}_{B_{\RR_0}\left(\overline\theta\right)}((s+t_1)\rrho+\theta)
\exp\left(\int_s^{t-t_1}\!
\d_x F_\beta\left((\tau+t_1)\rrho+\theta,\xi_\beta(\tau+t_1,\theta,x)\right)\, d\tau\right)
\, ds 
\\
& \geq\int_{0}^{t-t_1}\!
\mathbf{1}_{B_{\RR_0}\left(\overline\theta\right)\setminus B_{\tilde\RR}\left(\overline\theta\right)}((s+t_1)\rrho+\theta)
\exp\left(\int_s^{t-t_1}\!
\d_x F_\beta\left((\tau+t_1)\rrho+\theta,\xi_\beta(\tau+t_1,\theta,x)\right)\, d\tau\right)
\, ds 
\\
& \geq \left(\RR_0-\tilde\RR\right)/\rrho_D\cdot
\exp\left(-2b \int_{1/\rrho_D-\delta_2-t_1}^{t-t_1}\!
\xi_\beta(\tau+t_1,\theta,x)\, d\tau\right)
\geq 
\left(\RR_0-\tilde\RR\right)/\rrho_D\cdot
\exp\left(3/4\cdot b\eee\right)
\end{align*}
which is clearly bigger than $\exp(b\eee/2)$ for large enough $b$.

For the last relation, note that since $t_1\leq 1/\rrho_D-5\ee$, we have for $t\geq1/\rrho_D-\ee$ that
\begin{align*}
&\int_0^{t-t_1}\!
\d_x F_\beta\left((\tau+t_1)\rrho+\theta,\xi_\beta(\tau+t_1,\theta,x)\right)\, d\tau
\\
&=
\int_0^{1/\rrho_D-\ee-t_1}\!
\d_x F_\beta\left((\tau+t_1)\rrho+\theta,\xi_\beta(\tau+t_1,\theta,x)\right)\, d\tau
\\
&\phantom{\leq}+
\int_{1/\rrho_D-\ee-t_1}^{t-t_1}\!
\d_x F_\beta\left((\tau+t_1)\rrho+\theta,\xi_\beta(\tau+t_1,\theta,x)\right)\, d\tau\leq
-8b\ee\cdot(1-c)+4b\ee<0
\end{align*}
since $c<1/4$. Note that if $t<1/\rrho_D-\ee$, then \eqref{eq: dx xi <= 1} is obvious.
\end{proof}
\begin{proof}[Proof of Lemma~\ref{lem: nondegenerate bump}]
Let us fix some notation.
For the rest of this proof, $\vartheta$ is always assumed to be some element of $\R^d$ (the tangent space of $\T^d$ at any $\theta\in \T^d$) with $|\vartheta|=1$.
In contrast, $\Delta$ and $\Delta'$ always denote elements of $\mathbb S^d\ssq \R^D$ in the orthogonal complement of $\rrho$ in the following.
Set $T_\tau=\{\theta_0+\tau\rrho+\varepsilon\Delta\:\Delta \perp \rrho,\ |\Delta|=1 \text{ and }|\varepsilon|\leq\RR\}$ where $\tau\in(0,1/\rrho_D)$, that is,
$T_\tau$ is a $d$-dimensional disk of radius $\RR$ orthogonal to $\rrho$, centred at $\theta_0+\tau\rrho$. 
Similarly, set $\tilde T_\tau=\{\theta_0+\tau\rrho+\varepsilon\Delta\:\Delta \perp \rrho,\ |\Delta|=1 \text{ and }|\varepsilon|\leq\rr\}$ where
$\rr=\exp(-9b\ee)$. 
We denote by $P\theta$ the \emph{orthogonal projection} of $\theta\in\T^D$ onto $[\theta_0,\theta_0+\w]$ so that $\theta=P\theta+\eps \Delta$,
where $|\eps|$ is minimal
(with $\Delta\perp \rrho$ and $|\Delta|=1$ as above).\footnote{In the following, we only consider $P\theta$ for $\theta$ close to $[\theta_0,\theta_0+\w]$ 
so that we do not have
to worry about the well-definition of $P$.}
Set $t_1=1/(4\rrho_D)$ and note that--if $\RR$ is small enough--$T_{t_1}\cap \T^d=\emptyset$ and $[\I_0,T_{t_1}]\cap
B_\RR\left(\overline\theta\right)=\emptyset$.
Let $t_2$ be such that $T_{t_2}$ has a positive distance to $B_\RR\left(\overline \theta\right)$ and such that $T_{t_2}\cap \T^d=\emptyset$ 
and $[\I_0,T_{t_2}]\cap B_\RR\left(\overline \theta\right)=B_\RR\left(\overline \theta\right)$. 
Again, we might have to assume small enough $\RR$ in order for such $t_2$ to exist.
Finally, $t_3>t_2$ with $\tilde T_{t_3}\cap \T^d=\emptyset$ will be chosen to be close 
to $1/\rrho_D$ so that
within one iteration, orbits starting in $\tilde{\mc I}_0=\{\theta\in \T^d\: [\theta,\theta+\w]\cap {\tilde T}_{t_1}\neq \emptyset\}$ enter 
and leave the bump between $\tilde T_{t_1}$ and $\tilde T_{t_3}$ while the remaining time between $\tilde T_{t_3}$ and $\T^d$ will be negligibly short.
We let $t_i(\theta)\in[0,1/\rrho_D]$ be such that $\theta+t_i(\theta)\rrho \in T_{t_i}$ for $i=1,2$ and $\theta\in \I_0$. 
Likewise for $\theta\in \tilde \I_0$, we denote by $t_3(\theta)\in[0,1/\rrho_D]$ that time for which
$\theta+t_3(\theta)\rrho \in \tilde T_{t_3}$.
By considering small enough $\RR$, we may assume without loss of generality that $t_2(\theta)>1/\rrho_D-\eee/2$ for each $\theta\in \tilde \I_0$.
Note that the $t_i$'s are (restrictions of) affine linear maps whose derivatives are given by a constant matrix whose 
norm we denote by $\kappa$ for the rest of this proof (note that obviously $\textrm{d}t_1(\theta)=\textrm{d}t_2(\theta)=\textrm{d}t_3(\theta)$, 
where $\textrm{d}$ denotes the total derivative). 

The Hessian of $g(\theta) =h(|\theta-\overline\theta|)$ is easily seen to be
\begin{align*}
\textrm{d}^2 g(\theta)
=\textrm{d}\left(\frac{h'(|\theta-\overline\theta|)}{|\theta-\overline\theta|}(\theta-\overline\theta)\right)
=\left(\frac{h''(|\theta-\overline\theta|)}{|\theta-\overline\theta|^2}-
\frac{h'(|\theta-\overline\theta|)}{|\theta-\overline\theta|^3}\right)(\theta-\overline\theta)\cdot(\theta-\overline\theta)^\top+\frac{h'(|\theta-\overline\theta|)}{|\theta-\overline\theta|} I_D,
\end{align*}
where $I_D$ denotes the $D$-dimensional unit matrix. 
Hence for $\theta=P\theta+\eps \Delta'$, we have 
\begin{align}
\label{eq: d Delta g}
\d_\Delta g(\theta)&=\eps  \cdot \frac{h'(|\theta-\overline\theta|)}{|\theta-\overline\theta|} \langle \Delta',\Delta \rangle
\quad \text{and}\\
\label{eq: d 2 Delta g}
\d^2_\Delta g(\theta)&=\eps^2\cdot \left(\frac{h''(|\theta-\overline\theta|)}{|\theta-\overline\theta|^2}-
\frac{h'(|\theta-\overline\theta|)}{|\theta-\overline\theta|^3}\right){\langle \Delta',\Delta \rangle}^2+\frac{h'(|\theta-\overline\theta|)}{|\theta-\overline\theta|},
\end{align}
where $\langle\cdot,\cdot\rangle$ denotes the standard inner product in $\R^D$.

Having in mind \eqref{eq: d2 theta xi}, we see that in order to show the positivity of the second derivatives of $\xi_\beta$ 
with respect to $\vartheta \in \mathbb S^{d-1}$, we need  small enough upper bounds on the respective first derivatives in order to ensure that
the leading term under the integral is the one containing $\d^2_\vartheta F_\beta$.
To that end, we divide the iteration of an orbit starting at $(\theta,x)\in\tilde\I_0\times C$ into three time intervals
(see Figure~\ref{fig: subdivision of an iteration}).
\begin{figure}
\centering 
\includegraphics[scale=0.6]{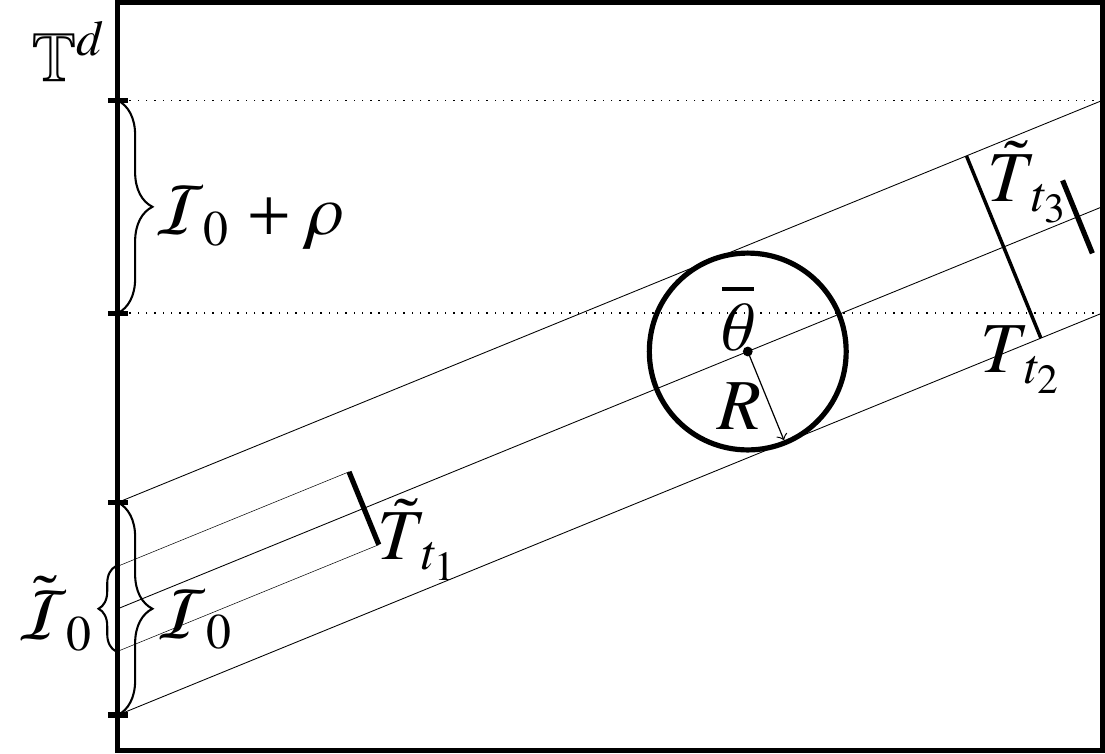}
\caption{The base space for $D=2$. 
We subdivide one iteration into three subsequent iterations: first, from $\tilde \I_0\ssq\T^d$ to $\tilde T_{t_1}$. 
Then, further to $\tilde T_{t_3}$. 
Finally, from $\tilde T_{t_3}$ to $\tilde\I_0+\w\ssq \T^d$. 
If $\theta_0+\tau\rrho\in B_\RR\left(\overline\theta\right)$, the disks $\tilde T_\tau$ are sections of the tangents of the level sets of $g$ 
at $\theta_0+\tau\rrho$.} 
\label{fig: subdivision of an iteration}
\end{figure}
Within the first interval, $[0,t_1(\theta)]$, variation with respect to $\theta$ only occurs due to the $\theta$-dependence of $t_1(\theta)$
which turns out to be negligible. 
The last time interval, $[t_3(\theta),1/\rrho_D]$, will turn out to be negligible as we can assume its length to be small.
For the intermediate time interval $[t_1(\theta),t_3(\theta)]$, the crucial point is that by the choice of the sets $\tilde T_\tau$ 
perpendicular to $\rrho$ and hence parallel to the level sets of $g$ at the point $\theta_0+\tau\rrho$, the derivatives with respect to $\vartheta$
become derivatives with respect to some $\Delta\perp\rrho$.
By \eqref{eq: d Delta g}, this implies that in a distance $\eps=\rr$ of $\theta_0+\tau\rrho$ (where $\tau\in[t_1(\theta),t_3(\theta)]$),
these derivatives are exponentially small in $b$ (recall that $\rr=\exp[-9b\delta_1]$).

While the first derivatives with respect to $\vartheta$ are thus negligible, we will show in Claim~\ref{claim: second derivatives} 
that $(\d^2_\Delta\xi_\beta)(\tau-t_1,\theta+t_1(\theta)\rrho,\xi_\beta(t_1(\theta),\theta,x))$ is bounded away from $0$ for each $\tau\in[t_2,t_3]$,
provided $\xi_\beta(1/\rrho_D,\theta,x)$ is not too far from $E$. 
In conclusion, we will show that $\d^2_\vartheta\xi_\beta(1/\rrho_D,\theta,x)$ is bounded away from $0$.
Together with the next claim, this will finish the proof of Lemma~\ref{lem: nondegenerate bump}.
\begin{claim}\label{claim: u(theta)}
Under the assumption of Lemma~\ref{lem: nondegenerate bump}, 
suppose there is $c_0>0$ (independent of $b$) such that $({\d^2_\Delta\xi_\beta})(t_2-t_1,\theta+t_1(\theta)\rrho,\xi_\beta(t_1(\theta),\theta,x))>c_0$ 
for all $\theta\in \tilde \I_0$ and $x\in C$ with ${\tilde\xi}_{\beta,\theta}(x)\leq-3/4$.

Then there is a closed and convex set $\mc J_{0,\beta}\ssq{\tilde \I}_{0}$  such that $\tilde\xi_{\beta,\theta}(1-c)>-1+\exp(-b/(2\rrho_D))$ if $\theta\notin\mc J_{0,\beta}$. Further, ${\tilde\xi}_{\beta,\theta}(x)\leq-3/4$ for each $\theta\in \mc J_{0,\beta}$ and $x\in C$, and $\mc J_{0,\beta}\ssq\mc J_{0,\beta'}$ for $\beta\leq \beta'\in[\beta_-,\beta_+]$.
\end{claim}
\begin{proof}[Proof of the claim]
For the rest of this proof, given $\theta\in T_{t_2}$, we denote by $\theta'$ that point in $\I_0$ for which 
$\theta=\theta'+t_2(\theta')\rrho$.

The map $u\:T_{t_2}\ni\theta\mapsto \xi_\beta(t_2(\theta'),\theta',1)$--we keep the dependence of $u$ on $\beta$
implicit--assumes its minimum in $\theta_0+t_2\rrho$ and moreover satisfies 
\begin{align}\label{eq: bump over T t2}
u(\theta)=\hat u(|\theta-(\theta_0+t_2\rrho)|),
\end{align} 
where $\hat u\:[0,\RR]\to\X$ is some non-decreasing function. This can be seen as follows:
First, we see that $\xi_\beta(t_1(\theta'),\theta',1)=1$ for each $\theta'\in\I_0$ since $F_\beta(\theta'+\tau\rrho,1)=0$
for all $\tau\in[0,t_1(\theta')]$ by definition of $t_1$. 
Hence, $u(\theta)=\xi_\beta(t_2-t_1,\theta-(t_2-t_1)\rrho,1)$.
Now note that for $\tau\in[0,t_2-t_1]$, we have
\begin{align*}
|\theta-(t_2-t_1)\rrho+\tau\rrho-\overline\theta|^2
&=|\theta-(t_2-t_1)\rrho+\tau\rrho-(\theta_0+(t_1+\tau)\rrho)|^2+
|\theta_0+(t_1+\tau)\rrho-\overline\theta|^2\\
&=|\theta-(\theta_0+t_2\rrho)|^2+
|\theta_0+(t_1+\tau)\rrho-\overline\theta|^2.
\end{align*}
Since $g(\cdot)=h(|(\cdot)-\overline\theta|)$, we therefore have
that there is $\hat F\: T_{t_2}\times [0,t_2-t_1]\times \X\to\R$ with $F_\beta(\theta-(t_2-t_1)\rrho+\tau\rrho,x)=\hat F(|\theta-(\theta_0+t_2\rrho)|,\tau,x)$, 
where $\hat F$ is non-decreasing in the first coordinate.
This proves \eqref{eq: bump over T t2}.

Set
\begin{align*}
 \mc J_{0,\beta}=\left\{\theta'\in { \I}_0\: u(\theta)\leq -1+\exp(-b/(2\rrho_D))+1/2 \cdot\exp(-b/\rrho_D)\right\}.
\end{align*}
Obviously, $\mc J_{0,\beta}$ is closed and $\mc J_{0,\beta}\ssq \mc J_{0,\beta'}$ for $\beta'\geq\beta$.
The convexity of $\mc J_{0,\beta}$ follows from \eqref{eq: bump over T t2}.
It hence remains to show that for sufficiently large $b$ we have $\mc J_{0,\beta}\ssq{\tilde \I}_{0}$, $\tilde\xi_{\beta,\theta}(1-c)>-1+\exp(-b/(2\rho_D))$ for $\theta'\notin\mc J_{0,\beta}$, and ${\tilde\xi}_{\beta,\theta'}(x)\leq-3/4$ for all $\theta'\in\mc J_{0,\beta}$ and $x\in C$.

First, we show $\tilde\xi_{\beta,\theta}(1-c)>-1+\exp(-b/(2\rho_D))$ for $\theta'\notin\mc J_{0,\beta}$.
Obviously, 
\[\xi_\beta(1/\rrho_D,\theta',1-c)\leq-1+\exp[-b/(2\rrho_D)]\]
if and only if
\[\xi_\beta(t_2(\theta'),\theta',1-c)\leq\xi_\beta^-(1/\rrho_D-t_2(\theta'),\theta'+\w,-1+\exp[-b/(2\rrho_D)])\] where $\xi_\beta^-(t,\theta',x)=\xi_\beta(-t,\theta',x)$.
Similarly to Proposition~\ref{prop: expansion in E contraction in C} \eqref{label: claim a}, we get for each $x\in C$ that  $\d_x\xi_\beta(t_2(\theta'),\theta',x)\leq
 \exp(-2b(1-c)(1/\rrho_D-\ee)+4b\ee)$, which is smaller than $\exp(-b/\rrho_D)$, since $\ee<1/(36\rrho_D)$ and $c<1/4$.
Therefore, 
\begin{align}\label{eq: width of Xi(C)}
|\xi_\beta(t_2(\theta'),\theta',1+c)-\xi_\beta(t_2(\theta'),\theta',1-c)|\leq 2c \exp(-b/\rrho_D) \leq 1/2 \exp(-b/\rrho_D). 
\end{align}
In particular, this implies
$|u(\theta)-\xi_\beta(t_2(\theta'),\theta',1-c)|< 1/2 \cdot\exp(-b/\rrho_D)$.
As further $\xi_\beta^-(1/\rrho_D-t_2(\theta'),\theta'+\w,-1+\exp[-b/(2\rrho_D)])\leq-1+\exp[-b/(2\rrho_D)]$ 
(due to Proposition~\ref{prop: radial basic observations} \eqref{label b prop radial basic observations}), this yields that
$\xi_\beta(t_2(\theta'),\theta',1-c)>\xi_\beta^-(1/\rrho_D-t_2(\theta'),\theta'+\w,-1+\exp[-b/(2\rrho_D)])$ if 
\begin{align}\label{eq: u > -1+e}
u(\theta)>-1+\exp(-b/(2\rrho_D))+1/2 \exp(-b/\rrho_D). 
\end{align}
Hence, $\tilde\xi_{\beta,\theta}(1-c)>-1+\exp(-b/(2\rho_D))$ for $\theta'\notin\mc J_{0,\beta}$.

Given $\theta'\in{\I}_0$ with $u(\theta)=\xi_\beta(t_2(\theta'),\theta',1)\geq-1$,
there is $y\in[-1,u(\theta)]$ such that
\begin{align*}
\tilde\xi_{\beta,\theta'}(1)
&=\xi_\beta(1/\rho_D-t_2(\theta'),\theta,-1)+
\d_x\xi_\beta(1/\rho_D-t_2(\theta'),\theta,y)\cdot|-1-u(\theta)|\\
&\leq -1+\exp(b\delta_2)\cdot|-1-u(\theta)|
\end{align*}
where we used \eqref{eq: d x xi} (recall that $t_2(\theta)>1/\rho_D-\delta_2/2$) and the fact that $\xi_\beta(1/\rho_D-t_2(\theta'),\theta,-1)=-1$.
Thus, ${\tilde\xi}_{\beta,\theta'}(1)>-7/8$ necessarily means 
$u(\theta)\geq-1+1/8 \exp(-b\delta_2)$ which is bigger than
the right-hand side of \eqref{eq: u > -1+e}
for large enough $b$ as $\eee<\ee\leq1/(36\rho_D)$.
Hence, ${\tilde\xi}_{\beta,\theta'}(x)\leq-3/4$ for all $\theta'\in\mc J_{0,\beta}$ and $x\in C$.

We are left to show that $\mc J_{0,\beta}\ssq{\tilde \I}_{0,\beta}$,
which is equivalent to showing that \eqref{eq: u > -1+e} holds for each $\theta\in T_{t_2}\setminus\tilde T_{t_2}$.
By the above, we may assume without loss of generality that
${\tilde\xi}_{\beta,\theta'}(1)\leq-3/4$ for all $\theta'\in{\tilde\I}_0$ so that $({\d^2_\Delta\xi_\beta})(t_2-t_1,\theta'+t_1(\theta')\rrho,1)>c_0$ by the hypothesis of this claim.
Note that by definition of $\beta_+$ and due to Proposition~\ref{prop: radial basic observations} \eqref{label a prop radial basic observations}, it follows from \eqref{eq: width of Xi(C)} that
$u(\theta)\geq u(\theta_0+t_2\rrho)\geq -1-1/2 \exp(-b/\rrho_D)$. 
Hence, for $\theta$ on the boundary of $\tilde T_{t_2}$, we get by means of the lower bound $c_0$ on the second derivatives that
\begin{align*}
 u(\theta)&\geq u(\theta_0+t_2\rrho)+ c_0 \cdot |\theta-(\theta_0+t_2\rrho)|^2\geq
-1-1/2 \exp(-b/\rrho_D)+c_0 r_b^2\\
&=-1-1/2 \exp(-b/\rrho_D)+c_0 \exp(-18b\ee)\\
&>
-1+\exp[-b/(2\rrho_D)]+1/2 \exp(-b/\rrho_D)
\end{align*}
for large enough $b$ as $\ee<1/(36\rrho_D)$.
By means of \eqref{eq: bump over T t2}, this proves \eqref{eq: u > -1+e} for all $\theta\in T_{t_2}\setminus\tilde T_{t_2}$.
\end{proof}
It remains to compute upper bounds on the first derivatives $\d_\vartheta\xi_\beta$ 
and lower bounds for the second derivatives $\d^2_\vartheta\xi_\beta$.
For $\theta\in \tilde \I_0$ and $x \in C$, we have
\begin{align}\label{eq: d vartheta xi with vartheta in T d}
\begin{split} 
\left|\d_\vartheta \xi_\beta(t_1(\theta),\theta,x)\right|
&\leq 
\left|(\d_\vartheta \xi_\beta)(t_1(\theta),\theta,x)\right|+\left|\d_t \xi_\beta( t_1(\theta),\theta,x)\cdot  \d_\vartheta t_1(\theta)\right|
\\&
\leq\kappa (1+c)^2 b.
\end{split}
\end{align}
This is due to the fact that $\left(\d_\vartheta \xi_\beta\right)(t_1(\theta),\theta,x)=0$ (see \eqref{eq: d theta xi} 
and recall that $[\tilde \I_0,\tilde T_{t_1}]\cap B_\RR(\overline\theta)=\emptyset$) and 
because $\xi_\beta( t_1(\theta),\theta,x)\in C$ for all $(\theta,x)\in \T^d\times C$ such that
\begin{align}\label{eq: d t xi beta( t1(theta),theta,x)} 
|\d_t \xi_\beta( t_1(\theta),\theta,x)|&=|F_\beta(\theta+t_1(\theta)\rrho,\xi_\beta[t_1(\theta),\theta,x])|\leq (1+c)^2 b.
\end{align}

For $t\in[t_1,t_3]$, $\theta\in\tilde\I_0$ and $x\in C$, we further have
\begin{align}\label{eq: d vartheta' xi}
\begin{split}
&\left|\left(\d_{\Delta} \xi_\beta\right)\left(t-t_1,\theta+ t_1(\theta)\rrho,\xi_\beta(t_1(\theta),\theta,x)\right)\right|\\
&\leq
\int_{0}^{t-t_1} \left|(\d_{\Delta} F_\beta)\left(\theta+[s+t_1(\theta)]\rrho, \xi_\beta[s+t_1(\theta),\theta,x]\right)\right|\\
&\phantom{\leq}\cdot \exp\left(\int_s^{t-t_1}\!
(\d_x F_\beta)\left(\theta+[\tau+t_1(\theta)]\rrho, \xi_\beta[\tau+t_1(\theta),\theta,x]\right)\, d\tau\right)
\ ds\\
&\leq \iota \cdot h''(0) \cdot b/(1-b^{-1/2}) \rr
\\
&\phantom{\leq}\cdot \int_0^{t-t_1}
\exp\left(\int_s^{t-t_1}\!
(\d_x F_\beta)\left(\theta+[\tau+t_1(\theta)]\rrho, \xi_\beta[\tau+t_1(\theta),\theta,x]\right)\, d\tau\right)
\, ds\\
&\leq\ \iota\cdot h''(0)\cdot b/(1-b^{-1/2}) \exp(5b\ee) \rr\leq
\exp(6b\ee) \rr 
\end{split}
\end{align}
for sufficiently large $b$, where we used 
\eqref{eq: d Delta g} in the second step (with $\rrr$ such that $|h'(y)/y|\leq \rrr| h''(0)|$ for all $y\geq0$) and
\eqref{eq: upper bound exp int} in the second to the last step. 
Observe that \eqref{eq: d vartheta' xi} is an upper bound on $|(\d_{\Delta} \xi_\beta )(t-t_1,\theta+ t_1(\theta)\rrho,x )|$ for 
\emph{all} $\Delta\perp\rrho$ of length $1$.

Now, the derivative of the map
$\tilde\I_0\times C \ni(\theta,x)\mapsto\xi_\beta\left(t-t_1+t_1(\theta),\theta,x\right)$
in direction of an arbitrary $\vartheta\in \R^d$ with $|\vartheta|=1$ is given by
\begin{align}\label{eq: d vartheta xi on the interesting section}
\begin{split}
 &\d_\vartheta \xi_\beta\left(t-t_1+t_1(\theta),\theta,x\right)=
 \d_\vartheta \xi_\beta\left(t-t_1,\theta+ t_1(\theta)\rrho,\xi_\beta( t_1(\theta),\theta,x)\right)\\
&=
\left(\textrm{d}_{\theta} \xi_\beta\right)\left(t-t_1,\theta+ t_1(\theta)\rrho,\xi_\beta( t_1(\theta),\theta,x)\right)\cdot
(\vartheta+\d_\vartheta t_1(\theta)\rrho)
\\
&\phantom{=}+
 \left(\d_x \xi_\beta\right)\left(t-t_1,\theta+t_1(\theta)\rrho,\xi_\beta( 
t_1(\theta),\theta,x)\right)\cdot 
\d_\vartheta \xi_\beta(t_1(\theta),\theta,x)
\\
&=
|\vartheta+\d_\vartheta t_1(\theta)\rrho| \cdot \left(\d_{\Delta} \xi_\beta\right)\left(t-t_1,\theta+t_1(\theta)\rrho,\xi_\beta(t_1(\theta),\theta,x)\right)
\\&\phantom{=}+
 \left(\d_x 
\xi_\beta\right)\left(t-t_1,\theta+t_1(\theta)\rrho,\xi_\beta(t_1(\theta),\theta,x)\right)\cdot \d_\vartheta \xi_\beta(t_1(\theta),\theta,x)
\end{split}
\end{align}
where $(\textrm{d}_\theta \xi_\beta)(t,\theta,x)$ denotes the total derivative of the map $\theta\mapsto\xi_\beta(t,\theta,x)$ 
(for fixed $t$ and $x$) and $\Delta= (\vartheta+\d_\vartheta t_1(\theta)\rrho)/|\vartheta+\d_\vartheta t_1(\theta)\rrho|$ is indeed orthogonal to $\rrho$.
Note that due to \eqref{eq: dx xi <= 1}, $(\d_x \xi_\beta)(t-t_1,\theta+t_1(\theta)\rrho,\xi_\beta[t_1(\theta),\theta,x])\leq1$ for all $t\in [t_1,1/\rrho_D]$
since $t_1(\theta)\leq t_1+\RR<t_1+\ee<1/\rrho_D-5\ee$ (recall that $t_1=1/(4\rrho_D)$).
By means of \eqref{eq: d vartheta xi with vartheta in T d} and \eqref{eq: d vartheta' xi}, we hence have
\begin{align}\label{eq: d vartheta xi on the interesting section estimate}
 \left|\d_\vartheta \xi_\beta\left(t-t_1+t_1(\theta),\theta,x\right)\right|\leq (1+\kappa|\rrho|) \exp(6b\ee) \rr+\kappa (1+c)^2 b.
\end{align}

We thus have upper bounds on the first derivatives of $\xi_\beta$ with respect to $\Delta$ and $\vartheta$.
We proceed with the second derivatives.
\begin{claim}\label{claim: second derivatives}
$({\d^2_\Delta\xi_\beta})(t-t_1,\theta+t_1(\theta)\rrho,\xi_\beta(t_1(\theta),\theta,x))>\exp(b\eee/2)$ 
 for all $\theta\in \tilde \I_0$, $t\in[t_2,t_3]$, and $x\in C$ with ${\tilde\xi}_{\beta,\theta}(x)\leq-3/4$.
\end{claim}
\begin{proof}[Proof of the claim]
As $h'\!\!\restriction_{(0,\RR)}<0$ and $\d_\Delta^2 g\left(\overline\theta\right)<0$, we see by means of
\eqref{eq: d 2 Delta g} that there is $\gamma_1>0$ such that for sufficiently large $b$ we have
$\d^2_\Delta g<-\gamma_1$ on $B_{\RR_0}\left(\overline\theta\right)\cap[\tilde \I_0,\tilde\I_0+\w]$, where 
$\RR_0>0$ is as in Claim~\ref{claim: int exp (int F(xi)) is small}.
Let 
\[
\gamma_2=\max_{\theta\in\T^D\setminus B_{\RR_0}\left(\overline\theta\right)} h''(|\theta-\overline\theta|)/|\theta-\overline\theta|^2-h'(|\theta-\overline\theta|)/|\theta-\overline\theta|^3\geq0
\]
and observe--again by means of \eqref{eq: d 2 Delta g}--that
$\d^2_\Delta g\leq\gamma_2\rr^2$ on
$\left(B_{\RR}\left(\overline\theta\right)\setminus B_{\RR_0}\left(\overline\theta\right)\right)
\cap[\tilde \I_0,\tilde\I_0+\w]$.

For $\theta$, $x$, and $t$ as in the hypothesis, we thus have
\begin{align}
\nonumber
& \int_0^{t-t_1}\!
(\d_\Delta^2 g)(\theta+[s+t_1(\theta)]\rrho)
\exp\left(\int_s^{t-t_1}\!
(\d_x F_\beta)\left(\theta+[\tau+t_1(\theta)]\rrho, \xi_\beta[\tau+t_1(\theta),\theta,x]\right)\, d\tau\right)
\, ds\\
\begin{split}\label{eq: upper bound on d 2 g}
&\leq\int_0^{t-t_1}\!
\left(\gamma_2 \rr^2 \mathbf{1}_{B_\RR\left(\overline\theta\right)\setminus B_{\RR_0}\left(\overline\theta\right)}(\theta+[s+t_1(\theta)]\rrho)-\gamma_1\mathbf{1}_{B_{\RR_0}\left(\overline\theta\right)}(\theta+[s+t_1(\theta)]\rrho)\right)
\\
&\phantom{=}\, \cdot\exp\left(\int_s^{t-t_1}\!
(\d_x F_\beta)\left(\theta+[\tau+t_1(\theta)]\rrho, \xi_\beta[\tau+t_1(\theta),\theta,x]\right)\, d\tau\right)
\, ds
\end{split}
\\
\nonumber
&\leq \gamma_2 \rr^2 \exp(5b\delta_1)-\gamma_1\exp(b\delta_2/2)\leq -\gamma_3\exp(b\delta_2/2)
\end{align}
for some $\gamma_3>0$, where we used \eqref{eq: upper bound exp int} and 
\eqref{eq: lower bound exp int} in the second to last step (recall that $\rr=\exp(-9b\ee)$).

Now, plugging \eqref{eq: d vartheta' xi} and \eqref{eq: upper bound exp int} into \eqref{eq: d2 theta xi} 
(observe that the term with the mixed derivatives of $F_\beta$ vanishes for \eqref{eq: radial x-square family}) yields for each $t\in[t_2,t_3]$ that
\begin{align*}
&\left|\left(\d_{\Delta}^2 \xi_\beta\right)\left(t-t_1,\theta+t_1(\theta)\rrho,\xi_\beta(t_1(\theta),\theta,x)\right)\right|
\\ 
&=\left|\int_0^{t-t_1}\!
\left[(\d^2_x F_\beta)(\theta+[s+t_1(\theta)]\rrho, 
\xi_\beta[s+t_1(\theta),\theta,x]) 
\left((\d_{\Delta} \xi_\beta)[s,\theta+t_1(\theta)\rho,\xi_\beta(t_1(\theta),\theta,x)]\right)^2
\right.\right.
\\
&\phantom{=}
\left.\vphantom{\left(\xi_\beta[s+t_1(\theta),\theta,x]\right)^2}
+(\d^2_{\Delta} F_\beta)\left(\theta+[s+t_1(\theta)]\rrho, \xi_\beta[s+t_1(\theta),\theta,x]\right)
\right]
\\
&\phantom{=}
\left.\left.\vphantom{\left(\xi_\beta[s,\theta+t_1(\theta)\rrho,x]\right)^2}\right.
\cdot
\exp\left(\int_s^{t-t_1}\!
(\d_x F_\beta)\left(\theta+[\tau+t_1(\theta)]\rrho, \xi_\beta[\tau+t_1(\theta),\theta,x]\right)\, d\tau\right)
\, ds \right |\\
&\geq 
-2 b \exp(17b\ee) \rr^2-\beta b/(1-b^{-1/2})
\\
&\phantom{=}
\cdot\!\int_0^{t-t_1}\!\!\!\!
(\d_\Delta^2 g)(\theta+[s+t_1(\theta)]\rrho)
\exp\left(\int_s^{t-t_1}\!\!\!
(\d_x F_\beta)\left(\theta+[\tau+t_1(\theta)]\rrho, \xi_\beta[\tau+t_1(\theta),\theta,x]\right)\, d\tau\right)
\, ds
\\
&\geq -2 b \exp(17b\ee) \rr^2+\gamma_3 \beta b/(1-b^{-1/2})\exp(b \eee/2)
\end{align*}
which is bigger than $\exp(b\eee/2)$ for sufficiently large $b$, where
 we used \eqref{eq: upper bound on d 2 g} in the last step.
\end{proof}
Thus, the assumptions of Claim~\ref{claim: u(theta)} are met and it remains to show that
$\d_\vartheta^2\tilde\xi_{\beta,\theta}(x)>\exp(b\delta_2/4)$ for $x\in C$, $\vartheta\in\mathbb S^{d-1}$, and $\theta\in\mc J_{0,\beta}$.
Plugging \eqref{eq: dx xi <= 1} into \eqref{eq: d2 x xi}, yields
\begin{align*}
\left|\left(\d_x^2 \xi_\beta\right)\left(t_3-t_1,\theta+ t_1(\theta)\rrho,\xi_\beta[t_1(\theta),\theta,x]\right)\right|\leq 2b/\rrho_D.
\end{align*}
Analogously, with 
\eqref{eq: dx dvartheta xi} and \eqref{eq: d vartheta' xi} we get
\begin{align*}
\left|\left(\d_{\Delta} \d_x \xi_\beta\right)\left(t_3-t_1,\theta+ t_1(\theta)\rrho,\xi_\beta[t_1(\theta),\theta,x]\right)\right|\leq 2b/\rrho_D\cdot \exp(6b\ee) \rr.
\end{align*}
Finally, note that
\begin{align}\label{eq: d2 vartheta xi with vartheta in T d}
 \d_\vartheta^2 \xi_\beta(t_1(\theta),\theta,x)
 =
\left(\d^2_\vartheta \xi_\beta\right)(t_1(\theta),\theta,x)+
2\left(\d_t\d_\vartheta \xi_\beta\right)(t_1(\theta),\theta,x)\cdot  \d_\vartheta t_1(\theta)
+\d^2_t \xi_\beta( t_1(\theta),\theta,x)\cdot  (\d_\vartheta t_1(\theta))^2,
\end{align}
where we used the fact that $\d^2_\vartheta t_1(\theta)=0$. By means of \eqref{eq: d theta xi}, we have that $\d_\vartheta \xi_\beta(\tau,\theta,x)=0$ for all $\tau \in [0,1/\rrho_D-\ee]$ so that both
$\left(\d^2_\vartheta \xi_\beta\right)(t_1(\theta),\theta,x)$ and
$\left(\d_t\d_\vartheta \xi_\beta\right)(t_1(\theta),\theta,x)\cdot  \d_\vartheta t_1(\theta)$ vanish.
Further,
\begin{align}\label{eq: d 2 t xi (t i (theta),...)}
\begin{split}
 \d^2_t \xi_\beta(t_1(\theta),\theta,x)&
=
\d_x F_\beta(\theta+t_1(\theta)\rrho,\xi_\beta(t_1(\theta),\theta,x))
\d_t\xi_\beta(t_1(\theta),\theta,x)
\\
&= -2b\cdot\xi_\beta(t_1(\theta),\theta,x)\cdot \d_t\xi_\beta(t_1(\theta),\theta,x)
\end{split}
\end{align}
where we used that $\textrm{d}_\theta F_{\beta}(\theta+t\rrho,x)=0$
for all $t\in[0,1/\rrho_D-\ee]$ in the first step.
Since $\xi_\beta(t_1(\theta),\theta,x)\leq1+c$ and due to \eqref{eq: d t xi beta( t1(theta),theta,x)}, we hence get
\begin{align}
 |\d_\vartheta^2 \xi_\beta(t_1(\theta),\theta,x)| \leq2 \kappa^2 (1+c)^3 b^2.
\end{align}

We are now in a position to derive a lower bound on the second derivative of 
$\tilde \I_0\times C \ni(\theta,x) \mapsto\xi_\beta(t_3(\theta),\theta,x)=\xi_\beta(t_3-t_1,\theta+t_1(\theta)\rrho,\xi_\beta[t_1(\theta),\theta,x])$ in direction of $\vartheta$. From \eqref{eq: d vartheta xi on the interesting section}, we get
\begin{align*}
 &\d_\vartheta^2 \xi_\beta\left(t_3(\theta),\theta,x\right)
\\
&=
|\vartheta+\d_\vartheta t_1(\theta)\rrho|^2 \cdot \left(\d_{\Delta}^2 \xi_\beta\right)\left(t_3-t_1,\theta+t_1(\theta)\rrho,\xi_\beta(t_1(\theta),\theta,x)\right)
\\
&\phantom{=} +2|\vartheta+\d_\vartheta t_1(\theta)\rrho|\left(\d_{\Delta}\d_x \xi_\beta\right)\left(t_3-t_1,\theta+t_1(\theta)\rrho,\xi_\beta( t_1(\theta),\theta,x)\right)\cdot 
\d_\vartheta \xi_\beta(t_1(\theta),\theta,x)
\\
&\phantom{=} +\left(\d_x^2 \xi_\beta\right)\left(t_3-t_1,\theta+ t_1(\theta)\rrho,\xi_\beta(t_1(\theta),\theta,x)\right)\cdot 
\left(\d_\vartheta \xi_\beta(t_1(\theta),\theta,x)\right)^2
\\
&\phantom{=} +
 \left(\d_x \xi_\beta\right)\left(t_3-t_1,\theta+t_1(\theta)\rrho,\xi_\beta( 
t_1(\theta),\theta,x)\right)\cdot 
\d^2_\vartheta \xi_\beta(t_1(\theta),\theta,x).
\end{align*}
By the above computations and in particular from Claim~\ref{claim: second derivatives}, we see that for large enough $b$ the leading term is the one
containing
$\left(\d_{\Delta}^2 \xi_\beta\right)\left(t_3-t_1,\theta+ t_1(\theta)\rrho,\xi_\beta(t_1(\theta),\theta,x)\right)$.
This yields
\begin{align}\label{eq: lower bound d vartheta 2 at t2}
 \left|\d_\vartheta^2 \xi_\beta\left(t_3(\theta),\theta,x\right)\right|\geq 
\exp(b\eee/3)
\end{align}
for large enough $b$.
Now, let us consider the derivatives $\d_\vartheta^2 \tilde \xi_{\beta,\theta}(x)$. 
Analogously to \eqref{eq: d vartheta xi on the interesting section}, we get
\begin{align*}
 \d_\vartheta \tilde \xi_{\beta,\theta}(x)&= \d_\vartheta \xi_{\beta}\left(1/\rrho_D-t_3(\theta),\theta+t_3(\theta)\rrho,\xi_\beta(t_3(\theta),\theta,x)\right)\\
&=
-\d_t \xi_{\beta}\left(1/\rrho_D-t_3(\theta),\theta+t_3(\theta)\rrho,\xi_\beta(t_3(\theta),\theta,x)\right)\cdot
\d_\vartheta t_3(\theta)
\\
&\phantom{=}+
|\vartheta+\d_\vartheta t_3(\theta)\rrho| \cdot \left(\d_{\Delta} \xi_\beta\right)\left(1/\rrho_D-t_3(\theta),\theta+t_3(\theta)\rrho,\xi_\beta(t_3(\theta),\theta,x)\right)
\\&\phantom{=}+
 \left(\d_x 
\xi_\beta\right)\left(1/\rrho_D-t_3(\theta),\theta+t_3(\theta)\rrho,\xi_\beta(t_3(\theta),\theta,x)\right)\cdot \d_\vartheta \xi_\beta(t_3(\theta),\theta,x)
\\&=
-\d_t \xi_{\beta}\left(1/\rrho_D-t_3(\theta),\theta+t_3(\theta)\rrho,\xi_\beta(t_3(\theta),\theta,x)\right)\cdot
\d_\vartheta t_3(\theta)
\\&\phantom{=}+
 \left(\d_x 
\xi_\beta\right)\left(1/\rrho_D-t_3(\theta),\theta+t_3(\theta)\rrho,\xi_\beta(t_3(\theta),\theta,x)\right)\cdot \d_\vartheta \xi_\beta(t_3(\theta),\theta,x),
\end{align*}
where we used that $F_\beta(\theta+t_3(\theta)\rrho+\tau,\cdot)=0$ for all $\tau\in[0,1/\rrho_D-t_3(\theta)]$ and $\theta\in\tilde \I_0$ in the last step.
By differentiating this expression once more, we straightforwardly obtain
\begin{align*}
\d_\vartheta^2 \tilde \xi_{\beta,\theta}(x)&=
\d_t^2 \xi_{\beta}\left(1/\rrho_D-t_3(\theta),\theta+t_3(\theta)\rrho,\xi_\beta(t_3(\theta),\theta,x)\right)\cdot
(\d_\vartheta t_3(\theta))^2\\
&\phantom{=}-
2(\d_t\d_x \xi_{\beta})\left(1/\rrho_D-t_3(\theta),\theta+t_3(\theta)\rrho,\xi_\beta(t_3(\theta),\theta,x)\right)\cdot
\d_\vartheta t_3(\theta)\cdot\d_\vartheta \xi_\beta(t_3(\theta),\theta,x)
\\&\phantom{=}+
 \left(\d_x^2 
\xi_\beta\right)\left(1/\rrho_D-t_3(\theta),\theta+t_3(\theta)\rrho,\xi_\beta(t_3(\theta),\theta,x)\right)\cdot \left(\d_\vartheta \xi_\beta(t_3(\theta),\theta,x)\right)^2
\\&\phantom{=}+
 \left(\d_x 
\xi_\beta\right)\left(1/\rrho_D-t_3(\theta),\theta+t_3(\theta)\rrho,\xi_\beta(t_3(\theta),\theta,x)\right)\cdot \d_\vartheta^2 \xi_\beta(t_3(\theta),\theta,x).                                      
\end{align*}
Let us discuss why 
$ (\d_x 
\xi_\beta)(1/\rrho_D-t_3(\theta),\theta+t_3(\theta)\rrho,\xi_\beta[t_3(\theta),\theta,x])\cdot \d_\vartheta^2 \xi_\beta(t_3(\theta),\theta,x)$ is the leading term. 
To that end, note that since $\xi_\beta(\tau,\theta+t_3(\theta)\rrho,\xi_\beta[t_3(\theta),\theta,x])<0$ for all $\tau\in[0,1/\rrho_D-t_3(\theta)]$ and $\theta\in \mc J_{0,\beta}$, we have $ (\d_x \xi_\beta)(1/\rrho_D-t_3(\theta),\theta+t_3(\theta)\rrho,\xi_\beta(t_3(\theta),\theta,x))\geq1$.
Together with \eqref{eq: lower bound d vartheta 2 at t2}, this eventually finishes the proof if we can show that the remaining terms are indeed negligible.

By an analogous computation as in \eqref{eq: d 2 t xi (t i (theta),...)}, we see
\begin{align*}
 \left|\d_t^2 \xi_{\beta}\left(1/\rrho_D-t_3(\theta),\theta+t_3(\theta)\rrho,\xi_\beta(t_3(\theta),\theta,x)\right)\right|&\leq 2b\cdot|\xi_\beta(1/\rrho_D,\theta,x)\cdot \d_t\xi_\beta(1/\rrho_D,\theta,x)|\\
&\leq 16b^2,
\end{align*}
where we used that
$\xi_{\beta}(1/\rrho_D,\theta,x)\geq-2$ (see Proposition~\ref{prop: expansion in E contraction in C} \eqref{label: claim a}) in the last step.
Further,
\begin{align*}
\left(\d_x 
\xi_\beta\right)\left(1/\rrho_D-t_3(\theta),\theta+t_3(\theta)\rrho,\xi_\beta(t_3(\theta),\theta,x)\right)\leq \exp(4b(1/\rrho_D-t_3))
\end{align*}
so that by putting $t_3$ close enough to $1/\rrho_D$ (which is possible if we assume large enough $b$) we get small enough upper bounds on $(\d_t\d_x \xi_{\beta})(1/\rrho_D-t_3(\theta),\theta+t_3(\theta)\rrho,\xi_\beta(t_3(\theta),\theta,x))$
(see \eqref{eq: d x fibre ode}) as well as
$(\d^2_x \xi_\beta)(1/\rrho_D-t_3(\theta),\theta+t_3(\theta)\rrho,
\xi_\beta(t_3(\theta),\theta,x))\cdot (\d_\vartheta \xi_\beta(t_3(\theta),\theta,x))^2$
(see \eqref{eq: d2 x xi} and \eqref{eq: d vartheta xi on the interesting section estimate}).
\end{proof}
There are six more assumptions on ${\tilde\Xi}_{\beta}$ to be considered.
These basically boil down to some weak upper bounds on further derivatives of the first return maps and their inverses.

Let $S>0$ be such that
\resume{enumerate}[{[${(\mathcal A}1)$]}]
\item $\left|\d_\vartheta {\tilde\xi}_{\beta,\theta}(x)\right|<S$
for all $\vartheta \in \mathbb S^{d-1}$ and $(\theta,x)\in\Gamma\cap {\tilde\Xi}_\beta^{-1}(\Gamma)$;
\label{axiom: 11c}
\item $\left|\d_\vartheta^2 {\tilde\xi}_{\beta,\theta}(x)\right|<S^2$
for all $\vartheta \in \mathbb S^{d-1}$ and $(\theta,x)\in\Gamma\cap {\tilde\Xi}_\beta^{-1}(\Gamma)$;
\label{axiom: 12c}
\item $ \left|\d_\vartheta\d_x {\tilde\xi}_{\beta,\theta}(x)\right|<
\begin{cases}
S \alpha_c & \text{for } (\theta,x) \in \T^d\times C \\
S\alpha_u^2 & \text{for } (\theta,x)\in\Gamma\cap {\tilde\Xi}_\beta^{-1}(\Gamma)
\end{cases}$ for each $\vartheta \in \mathbb S^{d-1}$. \label{axiom: 13c}
\suspend{enumerate}
Equations \eqref{eq: d theta xi}, \eqref{eq: d2 theta xi} and \eqref{eq: dx dvartheta xi} yield that a possible choice to ensure ${(\mathcal A}\ref{axiom: 11c})$--${(\mathcal A}\ref{axiom: 13c})$ for $\tilde\xi_{\beta,\theta}$ is to set $S=\exp(9b\ee)$. In case of
${(\mathcal A}\ref{axiom: 11c})$, this can be seen from
\begin{align*}
& \left|\d_\vartheta \tilde\xi_{\beta,\theta}(x)\right|\\
&\stackrel{\eqref{eq: d theta xi}}{\leq}
\int_0^{1/\rrho_D} \left|(\d_\vartheta F_\beta)\left(s\rrho+\theta, \xi_\beta(s,\theta,x)\right)\right|
\exp\left(\int_s^{1/\rrho_D}\!
(\d_x F_\beta)\left(\tau\rrho+\theta, \xi_\beta(\tau,\theta,x)\right)\, d\tau\right)
\, ds\\
&\stackrel{\phantom{\eqref{eq: d theta xi}}}{\leq} \ee b/(1-b^{-1/2})\cdot \max_{\theta\in\T^D}|\d_\vartheta g(\theta)|\exp\left(2b\ee\right),
\end{align*}
 where we used that $\d_\vartheta F_\beta$ vanishes for $s<1/\rrho_D-\ee$ and that $\xi_\beta(\tau,\theta,x)\geq-1$ for each $\tau\in[0,1/\rrho_D]$ and $(\theta,x)\in{\tilde\Xi}_\beta^{-1}(\Gamma)$ due to Proposition~\ref{prop: radial basic observations} (\ref{label a prop radial basic observations}).
 However, for big enough $b$, this expression is certainly smaller than $\exp(9b\ee)$. 
${(\mathcal A}\ref{axiom: 12c})$ and ${(\mathcal A}\ref{axiom: 13c})$ can be seen in a similar fashion.
Finally, we need that
\resume{enumerate}[{[${(\mathcal A}1)$]}]
\item $\left|\d_x^2 {\tilde\xi}_{\beta,\theta}(x)\right|<
\begin{cases}
\alpha_c & \text{for } (\theta,x)\in \T^d\times C \\
\alpha_u^2 & \text{for } (\theta,x)\in \Gamma\cap {\tilde\Xi}_\beta^{-1}(\Gamma),
\end{cases}$ \label{axiom: 14c}
\suspend{enumerate}
which is true due to, in particular, \eqref{eq: d2 x xi} and Proposition~\ref{prop: expansion in E contraction in C}.

There are two more assumptions left which deal with the inverse of ${\tilde\xi}_{\beta,\theta}$.
\resume{enumerate}[{[${(\mathcal A}1)$]}]
\item $\left|\d_x^2 {\tilde\xi}_{\beta,\theta}^{-1}(x)\right|< \alpha_e^{-1}$
for each
$\theta\notin \I_{0}+\w$ and $x\in E$;\label{axiom: 15c}
\item $\left|\d_\vartheta\d_x {\tilde\xi}_{\beta,\theta}^{-1}(x)\right|< S \alpha_e^{-1}$
for each $\theta\notin \I_{0}+\w, x\in E$ and $\vartheta \in \mathbb S^{d-1}$.\label{axiom: 16c}
\end{enumerate}
Observe that $\tilde \xi^{-1}_{\beta,\theta}=\tilde {\xi^-}_{\beta,\theta}$ (see \eqref{eq: xi inverse} and \eqref{eq: inverse non-autonomous vectorfield}). 
Hence, we can derive the desired estimates for $\tilde \xi^-_{\beta,\theta}$ by means of \eqref{eq: d2 x xi} and \eqref{eq: dx dvartheta xi} if we replace $F_\beta$ by $F_\beta^-$ and $\rrho$ by $\rrho^-=-\rrho$.
Under the assumption of $x\in E$ and $\theta\notin \I_{0,\beta}+\rrho$, we have
that $\xi^-_{\beta}(t,\theta,x)\in E$ for all $t\in[0,1/\rrho_D]$ and hence
$\d_x\xi^-_{\beta}(t,\theta,x)\leq \exp(-2b(1-\exp[-b/(2\rrho_D)])\cdot t)$. Thus, 
${(\mathcal A}\ref{axiom: 15c})$ follows immediately for large enough $b$.
${(\mathcal A}\ref{axiom: 16c})$ follows directly from the fact that $\d_\vartheta F^-_\beta(t\rrho^-+\theta,x)=0$ and hence $\d_\vartheta\xi^-_\beta(t,\theta,x)=0$ for 
$\theta\notin \I_{0,\beta}+\rrho$ and $t\in[0,1/\rrho_D]$.

\subsection{Occurrence of a non-smooth bifurcation}
We are now in a position to recast Theorem~\ref{thm: existence of sna discrete time} by spelling out the definition of $\mc U_\w(\X)$
in a way adapted to the first return maps ${\tilde\Xi}_{\beta}$ corresponding to \eqref{eq: radial x-square family}.

Given $\alpha,p>1$ and $K\in\N$, set $q=1-1/K$ and
\begin{align*}
 \nu=s-\kappa\left(\alpha,q\right)S^2\alpha^{-(2q^2/p-5(1-q^2)p)},
\end{align*}
where $\kappa(\alpha,q)$ is decreasing\footnote{Here, we only need the decreasing behaviour of the constant $\kappa$. 
The interested reader is referred to \cite[{Lemma~4.2.13}]{phdfuhrmann2015} for further details.}
in both $\alpha$ and $q$ and $s$ is the lower bound in $(\mc A\ref{axiom: 9c})$.
Our reformulation of Theorem~\ref{thm: existence of sna discrete time} reads as follows.
\begin{thm}[{cf. \cite[Theorem~4.18]{fuhrmann2014} and \cite[Theorem~4.2.15]{phdfuhrmann2015}}]\label{thm: sink-source orbit refined}
Suppose $\w\in \T^d$ is Diophantine of type $(\mathscr C',\eta)$, $\X\ssq\R$ is some 
non-degenerate interval and ${(\tilde\xi_\beta)}_{\beta\in [0,1]}$ lies in $\mc P_\w(\X)$ and satisfies $(\A\ref{axiom: 1c})$--$(\A\ref{axiom: 16c})$.
Let there be $p\geq \sqrt{2}$ and $\alpha>1$ with
\begin{align*}
 \alpha_c^{-1}=\alpha_e \geq \alpha^{2/p}, \qquad \alpha_l^{-1}=\alpha_u \leq \alpha^p.
\end{align*}
Further, assume $3|\I_0|<\mathscr C' (2KM)^{-\eta}$ for some integers $M$ not smaller than $2$ and $K$ such that $2q^2/p-5(1-q^2)p>0$ and assume 
$\nu>0$ as well as $\alpha>\alpha_0$, where $\alpha_0=\alpha_0(\nu,K,M,p,|C|,|E|,\eta,\mathscr C')$.
Then there is $\beta_c \in [0,1]$ 
such that $\tilde\xi_{\beta_c}$ has an SNA and an SNR.
\end{thm}
\begin{rem}
 Here, $|C|$ and $|E|$ denote the length of the intervals of contraction and expansion, respectively. 
It is important to mention that $\alpha_0$ can be chosen to be non-increasing in $\nu$ and non-decreasing in $|C|$ and $|E|$ for fixed $K$, $M$, $p$, $\eta$ and $\mathscr C'$.
\end{rem}

We want to show that ${({\tilde\Xi}_{\beta})}_{\beta\in[0,1]}$ verifies the hypothesis of Theorem~\ref{thm: sink-source orbit refined} if $R\geq \mathscr R$ (for some 
$\mathscr R=\mathscr R(|\rho|,\mathscr C,\eta)$) and $b$ is large enough.
It is straightforward to see that ${({\tilde\Xi}_{\beta})}_{\beta\in[0,1]}\in\mc P_\w(\R)$ and that $\w$ is Diophantine (cf. Section~\ref{sec: poincare map}).
Now, assume that $c$ and $\ee$ are small enough\footnote{In the case of $\ee$, this essentially amounts 
to assuming small enough $\RR$.} so that
\[2b(1-c) (1/\rrho_D-\ee)-10b\ee>b(1+c)/\rrho_D.\]
Then, setting $\alpha=\exp(b(1+c)/\rrho_D)$ and $p=2$ ensures 
$\alpha_c^{-1}=\alpha_e \geq \alpha^{2/p}$ and $\alpha_l^{-1}=\alpha_u = \alpha^p$.
We have just seen in this section that
$(\mc A\ref{axiom: 1c})$--$(\mc A\ref{axiom: 16c})$
are verified by ${({\tilde\Xi}_{\beta})}_{\beta\in[0,1]}$.
In fact, observe that $(\mc A\ref{axiom: 1c})$--$(\mc A\ref{axiom: 16c})$ still hold when we set the lower bound of the expanding interval $E$ to be $-1-\eps$ (for some sufficiently small $\eps=\eps(b)>0$) instead of $-1$.
Note further that we can choose $\alpha$ as big as we need by assuming large enough $b$.\footnote{Which amounts to assuming big $K$ and hence, small $\RR$.}

In Theorem~\ref{thm: sink-source orbit refined}, we moreover assume that
\begin{align}\label{eq: I 0 is small in continuous time}
3|\I_0|<\mathscr C' (2KM)^{-\eta} 
\end{align}
for some positive integers $M$ not smaller than $2$ and $K$ such that $q=1-1/K$ satisfies $2q^2/p-5(1-q^2)p=q^2-10(1-q^2)>0$.
Observe that--given some $M\in\N_{\geq 2}$ and such $K$--\eqref{eq: I 0 is small in continuous time} holds true under the assumption of small enough $\RR$ (independent of $b$).

Finally, we need
\begin{align*}
 \nu=s-\kappa\left(\alpha,q\right)S^2\alpha^{-(q^2-10(1-q^2))}
\end{align*}
to be positive.
Now, with $S$ as above and $s>\exp(b\eee/4)$ (cf. Lemma~\ref{lem: nondegenerate bump}), we get
\begin{align*}
 \nu> 
\exp(b\eee/4)-\kappa\left(\exp(b(1+c)/\rrho_D),q\right)
\exp\left(-b(1+c)[q^2-10(1-q^2)]/\rrho_D+18b\ee\right),
\end{align*}
which is positive (for sufficiently large $b$) and increasing in $b$ as long as $q$ is close to $1$ and hence, as long as $\RR$ is small.
Altogether, this shows: for big enough $R$ (independent of $b$ and $h$) and big enough $b$, ${({\tilde\Xi}_{\beta})}_{\beta\in[0,1]}$ verifies the assumptions of Theorem~\ref{thm: sink-source orbit refined}.

Let us fix a Diophantine $\rho\in\R^D$ and only consider families of flows $\hat{\Xi}$ driven by $(t,\theta)\mapsto t\cdot\rho+\theta$ in the following.
We define $\U_\rho(\X)$ to be the set of all $\hat F\in\P(\X)$ which generate families $\hat{\Xi}$ with $\hat{\tilde \Xi}\in \mc U_\w(\X)$, that is,
$\hat{\tilde\Xi}$ verifies the assumptions of Theorem~\ref{thm: sink-source orbit refined}.
From the above, we see that there exists $\hat F\in \U_\rho(\X)$ such that any $\mc C^2$-small perturbation of $\hat{\tilde \Xi}$ still lies in $\mc U_\w(\X)$.
Since $\mc C^2$-small changes of $\hat F$ (recall that we actually consider the modified
vector field [cf. Section~\ref{sec: poincare map}]) result in $\mc C^2$-small changes of $\hat{\tilde \Xi}$ \cite[\S12 Satz~\rom{6}]{walter1976}, this proves that 
$\mc C^2$-small changes of $\hat F\in\U_\rho(\X)$ still lie in $\U_\rho(\X)$.
In other words, Theorem~\ref{thm: main existence continuous time} holds true for $\X=\R$.
In fact, a straightforward adaption of \eqref{eq: radial x-square family} immediately yields 
Theorem~\ref{thm: main existence continuous time} for arbitrary non-degenerate intervals $\X\ssq\R$.

\subsection{Geometry of the invariant graphs}\label{sec: geometry}
To close the discussion of the continuous time case, let us see how Theorem~\ref{t.main discrete} extends to Theorem~\ref{t.main}.
We denote the boundary graphs of the maximal invariant set
$\tilde\Lambda_{\beta_c}$ of $\tilde \Xi_{\beta_c}$ by $\psi^\pm_{\tilde\Lambda_{\beta_c}}$ and those of the maximal invariant set $\Lambda_{\beta_c}$ of $\Xi_{\beta_c}$ by $\phi^\pm_{\Lambda_{\beta_c}}$.
Notice that 
\begin{align}\label{eq: relation invariant graphs discrete and cont time}
\Lambda_{\beta_c}=\Xi_{\beta_c}\left([0,1/\rho_D]\times\tilde\Lambda_{\beta_c}\right)
 \qquad \text{and} \qquad 
\Phi^\pm_{\Lambda_{\beta_c}}=\Xi_{\beta_c}\left([0,1/\rho_D]\times\Psi^\pm_{\tilde\Lambda_{\beta_c}}\right).
 \end{align}
We restrict to $\phi^+_{\Lambda_{\beta_c}}$ since $\phi^-_{\Lambda_{\beta_c}}$ can be dealt with similarly.
The uniqueness of the semi-continuous representatives of $\phi^+_{\Lambda_{\beta_c}}$
and item \eqref{item: 3 geometric continuous} are immediate.
Since $\phi^+_{\Lambda_{\beta_c}}$ and $\phi^-_{\Lambda_{\beta_c}}$ are $\textrm{Leb}_{\T^D}$-almost surely distinct, 
Lemma~\ref{lem: box dimension} gives
$D_B(\Phi^+_{\Lambda_{\beta_c}})=D_B\left(\overline{\Phi^+_{\Lambda_{\beta_c}}}\right)=D+1$.\footnote{We may view $\overline{\Phi^+_{\Lambda_{\beta_c}}}$ as a subset of $\R^{D+1}$.}

For the remaining properties, note that the remark in Section~\ref{sec: skew product maps} implies that
we just have to show that a statement similar to Proposition~\ref{prop: lipschitz} holds true in the continuous time case. 
Let $\tilde\Omega_j$ be as in Proposition~\ref{prop: lipschitz} and
set $\Omega_j=[\tilde\Omega_j,\tilde\Omega_j+\w]$ where $j\in\N$. Observe that 
\[\phi^+_{\Lambda_{\beta_c}}\!\!\restriction_{\Omega_j}\:\Omega_j\ni((\theta,0) +\theta_D/\rho_D\cdot \rho)\mapsto \Xi_{\beta_c}\left(\theta_D/\rho_D,(\theta,0),\psi^+_{\tilde\Lambda_{\beta_c}}(\theta)\right)\]
is Lipschitz continuous.
Finally, consider $\Omega_\infty=\T^D\setminus\bigcup_{j\in\N} \Omega_j\ssq [\tilde\Omega_\infty,\tilde\Omega_\infty+\w]$.
Due to Lemma~\ref{lem: lipschitz image hausdorff dimension}, Theorem~\ref{thm: Hausdorff dimension product sets}, and Proposition~\ref{prop: lipschitz}
\[D_H([\tilde\Omega_\infty,\tilde\Omega_\infty+\w])= D_H(\tilde\Omega_\infty\times[0,1/\rho_D]) \leq d.\]

Due to the monotonicity of the Hausdorff dimension, we hence get $D_H(\Omega_\infty)\leq D_H([\tilde\Omega_\infty,\tilde\Omega_\infty+\w])\leq d$ and thus
we have an analogue to Proposition~\ref{prop: lipschitz}.
This finishes the proof of Theorem~\ref{t.main}.

\bibliography{Literaturnachweis_SNA}{}
\bibliographystyle{plain}
\end{document}